\pdfoutput=1
\documentclass[twoside, preprint]{article}

\usepackage{aistats2027}

\usepackage[utf8]{inputenc}
\usepackage[T1]{fontenc}
\usepackage{amsmath}
\usepackage{amssymb}
\usepackage{amsthm}
\usepackage{amsfonts}
\usepackage{bbm}
\usepackage{enumitem}
\usepackage{array}
\usepackage{natbib}
\usepackage{xcolor}
\colorlet{RED}{red}
\definecolor{viridispurple}{HTML}{440154}
\definecolor{viridisblue}{HTML}{2D708E}
\definecolor{viridisteal}{HTML}{1F968B}
\definecolor{viridisgreen}{HTML}{35B779}
\usepackage{graphicx}
\usepackage[colorlinks=true, linkcolor=blue, citecolor=blue]{hyperref}
\usepackage{tikz}
\usetikzlibrary{arrows.meta,positioning,calc}
\usepackage{subcaption}
\newcounter{acond}
\renewcommand{\theacond}{A\arabic{acond}}
\newcommand{\Aconditem}{%
  \refstepcounter{acond}\item[\textbf{(\theacond)}]%
}
\theoremstyle{plain}
\newtheorem{theorem}{Theorem}[section]
\newtheorem{lemma}[theorem]{Lemma}
\newtheorem{corollary}[theorem]{Corollary}
\newtheorem{proposition}[theorem]{Proposition}
\theoremstyle{definition}
\newtheorem{definition}[theorem]{Definition}
\theoremstyle{remark}
\newtheorem{remark}[theorem]{Remark}
\newtheorem{example}[theorem]{Example}
\theoremstyle{plain}

\theoremstyle{remark}

\newcommand{\Pcal}{\mathcal{P}}
\newcommand{\Hcal}{\mathcal{H}}
\newcommand{\Wcal}{\mathcal{W}}
\newcommand{\E}{\mathbb{E}}
\newcommand{\Zcal}{\mathcal{Z}}
\DeclareMathOperator*{\argmin}{arg\,min}

\begin{document}

\twocolumn[

\aistatstitle{Beyond Discrepancy: Oracle Drift in Nonstationary Weighted Risk Minimization}

\aistatsauthor{Tobias Brock  \And Thomas Nagler}

\aistatsaddress{LMU Munich \\   Munich Center for Machine Learning}
]
\begin{abstract}
We analyze distributional drift in nonstationary weighted risk minimization by
decomposing the excess risk into estimation, drift, and approximation components. The drift error is commonly bounded through discrepancy-type quantities that control changes of the risk over classes of predictors. We instead relate the drift error directly to the movement of the best predictor within the model class and analyze its behavior in parametrized regression models. Our analysis shows that discrepancy-type quantities can be overly conservative: unlike our oracle drift, they may scale at a larger order than the resulting excess-risk contribution and can respond to changes in the risk surface that do not affect the optimal predictor. For linear regression, we characterize this gap sharply and combine the resulting drift bounds with estimation guarantees. We further extend the analysis to nonparametric regression.
\end{abstract}

\section{Introduction}

In time series, the random variables $X_1,\ldots,X_t$ are drawn from (possibly) different distributions $P_1,\ldots,P_t$.
We want to estimate a hypothesis $\hat h$ from past data $X_1,\ldots,X_{t-1}$ by empirical risk minimization (ERM) out of a class of functions $\Hcal$ and measure its performance at time $t$ by the excess risk $R_t(\hat h)-R_t(h_t^*)$, where $R_t(h)=\E_{P_t}[L(X,h)]$ and
\[
 h_t^*\in\argmin_{h\ \mathrm{measurable}}R_t(h),\qquad
 \bar h_t\in\argmin_{h\in\Hcal}R_t(h)
\]
denote the Bayes optimal predictor and the in-class oracle at time $t$, respectively.
The standard decomposition of the excess risk is given by
\[
 R_t(\hat h)-R_t(h_t^*)=
 \underbrace{R_t(\hat h)-R_t(\bar h_t)}_{\text{estimation}}
 +\underbrace{R_t(\bar h_t)-R_t(h_t^*)}_{\text{approximation}},
\]
where estimation depends on the data and the richness of $\Hcal$, and approximation captures the model misspecification.
While both of these quantities are well understood if the time series is independent and identically distributed (\textit{iid}), real data often exhibits nonstationarity or dependence, as for example in finance \citep{cont2001empirical}.
The analysis of the excess risk under nonstationarity is challenging because the distributions $P_1,\ldots,P_t$ may drift over time.

In such scenarios, it is natural to weight the risks.
Let $w\in[0,1]^{t-1}$ with $\sum_{i=1}^{t-1}w_i=1$. Define
\begin{align*}
 R_t^w(h)=\sum_{i=1}^{t-1}w_iL(X_i,h),\;
 R^w(h)=\sum_{i=1}^{t-1}w_i\E[L(X_i,h)],
\end{align*}
and denote the weighted empirical risk minimizer and the in-class weighted oracle as
\[
 \hat h_w=\argmin_{h\in\Hcal}R_t^w(h),\quad
 \bar h_w=\argmin_{h\in\Hcal}R^w(h).
\]
The choice of the weight family $\Wcal$ is highly task-specific and, e.g., determined via backtesting.

\paragraph{Related Work}
Previous works derive bounds for the generalization gap under drift using discrepancy-type quantities \citep{pmlr-v206-awasthi23b}, including path-dependent versions based on conditional expected losses \citep{KuznetsovMohri2016,KuznetsovMohri2017,KuznetsovMohri2020}, while \citet{HannekeYang2019} bound the in-class excess risk.
We use the symmetric marginal version
\[
    \mathrm{dis}(P_{i},P_j)
    =\sup_{h\in\Hcal}|R_i(h)-R_j(h)|,
\]
since comparing the two class oracles involves risk differences in both directions.
A simple add-and-subtract argument shows that the discrepancy is a natural upper bound of the movement of the in-class oracle,
\begin{align} \label{eq:decomposition_discrepancy}
R_t(\bar h_w)-R_t(\bar h_t) \leq 2
\mathrm{dis}\left(\sum_{i=1}^{t-1}w_iP_i,P_t\right).
\end{align}
It requires control over the entire hypothesis class.
Localized versions of the discrepancy restrict the comparison to a subset of the class \citep{zhang2020localized,demathelin2022discrepancy}.
More recently, \citet{brock2026fast} develop fast-rate oracle inequalities for weighted ERM, and decompose the excess risk into learning and drift errors.
Other work selects weights or windows adaptively \citep{huangWang2025stability,baby2025adaptive}, which is complementary to our analysis for a given class of weights $\Wcal$.
Appendix~\ref{sec:related_work} gives an extended overview.

\paragraph{Problem Formulation}
Our goal is to measure the performance of $\hat h_w$ at time $t$ by considering the in-class decomposition visualized in Figure~\ref{fig:two_decompositions} that separates the out-of-sample excess risk into estimation, drift, and approximation components:
\begin{align}\label{eq:decomposition_class_oracle_2}
&R_t(\hat h_w)-R_t(h_t^*)
=\underbrace{R_t(\hat h_w)-R_t(\bar h_w)}_{\text{estimation}}\notag\\
&\quad+\underbrace{R_t(\bar h_w)-R_t(\bar h_t)}_{\text{oracle drift}}
+\underbrace{R_t(\bar h_t)-R_t(h_t^*)}_{\text{approximation}}.
\end{align}
The estimation term can be controlled by a fast-rate oracle inequality that holds uniformly across distributions, while the approximation term is standard and can be bounded by established arguments.

The focus of this work is the analysis of the oracle drift that measures the movement of the best predictor available in the chosen class at time $t$ against the weighted in-class oracle of the past distributions,
and its relation to discrepancy-based quantities.
In particular, the oracle drift can be zero or higher-order when discrepancy-based upper bounds are large, even in standard settings.

\paragraph{Contribution}
Our contributions are as follows:
\begin{enumerate}[leftmargin=*,nosep]
 \item We bound the oracle drift with discrepancy-type quantities and discuss an alternative excess-risk decomposition (Section~\ref{sec:decompositions}).
 \item We state a general bound for the oracle drift for parametrized function classes and a fast-rate bound for the estimation term (Section~\ref{sec:estimation}).
 \item For linear regression, we show that the oracle drift scales optimally in the excess-risk contribution, whereas discrepancy-based measures can be overly conservative and respond to aspects of the data that are irrelevant to the optimal predictor (Section~\ref{sec:linear_models}).
 \item We combine the drift results with a fast-rate estimation bound to obtain explicit excess-risk guarantees for linear and nonparametric regression and discuss the bias--variance trade-off of estimation and drift (Sections~\ref{sec:linear_models}--\ref{sec:nonparametric-extension}).
\end{enumerate}

\begin{figure}[t]
\centering
\tikzset{
  >=Latex,
  decomp space/.style={draw=black!55, thick},
  decomp sample/.style={circle, fill=viridisteal!75!white, inner sep=1.05pt},
  decomp point/.style={circle, fill=black, inner sep=1.55pt},
  decomp est/.style={->, thick, viridispurple!78!white},
  decomp drift/.style={->, thick, viridisgreen!78!white},
  decomp inclass/.style={->, thick, viridisblue!78!white},
  hyplabel/.style={font=\footnotesize},
  errlabel/.style={font=\fontsize{5.5}{6}\selectfont}
}
\def\setshift{-0.18}
\resizebox{0.97\linewidth}{!}{%
\begin{tikzpicture}[font=\scriptsize]
  \draw[decomp space] (\setshift,0) ellipse [x radius=1.86, y radius=1.82];
  \node at ({-1.04+\setshift},0.82) {\Large $\Hcal$};
  \foreach \x/\y in {-1.68/0.10,-1.12/-0.95,-0.50/0.92,0.10/1.08}
    \node[decomp sample] at ({\x+\setshift},\y) {};
  \node[decomp point,label={[hyplabel]above:{$\hat h_w$}}] (A) at ({-1.25+\setshift},-0.05) {};
  \node[decomp point,label={[hyplabel]above:{$\bar h_w$}}] (B) at ({-0.10+\setshift},-0.05) {};
  \node[decomp point,label={[hyplabel]above:{$\bar h_t$}}] (C) at ({1.20+\setshift},-0.05) {};
  \node[decomp point,label={[hyplabel]above:{$h_t^*$}}] (D) at (3.85,-0.05) {};
  \draw[decomp est] (A) -- (B);
  \draw[decomp inclass] (B) -- (C);
  \draw[decomp drift] (C) -- (D);
  \node[viridispurple!78!white, errlabel, align=center] at (-0.86,-0.34) {estimation};
  \node[viridisblue!78!white, errlabel, align=center] at (0.37,-0.34) {oracle\\drift};
  \node[viridisgreen!78!white, errlabel, align=center] at (2.58,-0.34) {approximation};
\end{tikzpicture}%
}
\caption{Visualization of the in-class decomposition.}
\label{fig:two_decompositions}
\end{figure}
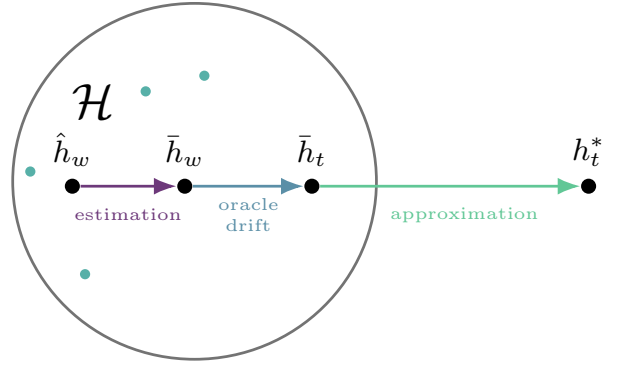

\section{Related Drift Measures}\label{sec:decompositions}
We now refine the discrepancy bound \eqref{eq:decomposition_discrepancy} by localization and contrast the oracle drift with the Bayes drift.

\paragraph{Localized Discrepancy}
The discrepancy can be quite conservative, as it takes a supremum over all of $\Hcal$, although predictors with large weighted risk are usually not selected by the oracle or by weighted ERM.
Its natural localized analogue retains the same risk difference but restricts the supremum to predictors near the weighted oracle.
For a radius $\varepsilon\ge0$, define
\[
 \Hcal_{w,\varepsilon}=\{h\in\Hcal:R^w(h)-R^w(\bar h_w)\le\varepsilon\}
\]
and, with $P_w=\sum_{i<t}w_iP_i$,
\begin{equation*}
 \mathrm{dis}_\varepsilon^{\mathrm{loc}}(P_w,P_t)
 =\sup_{h\in\Hcal_{w,\varepsilon}}
 |R^w(h)-R_t(h)|.
\end{equation*}
Localized generalization analyses typically choose $\varepsilon$ to contain $\hat h_w$, whereas the oracle comparison below also requires $\bar h_t\in\Hcal_{w,\varepsilon}$.

Clearly, $\mathrm{dis}_\varepsilon^{\mathrm{loc}}(P_w,P_t)
\le \mathrm{dis}(P_w,P_t)$. If
$\varepsilon\ge R^w(\bar h_t)-R^w(\bar h_w)$, so that
$\bar h_t\in\Hcal_{w,\varepsilon}$, the argument for
\eqref{eq:decomposition_discrepancy} applies on
$\Hcal_{w,\varepsilon}$ and gives
\begin{align}\label{eq:decomposition_localized}
\frac{1}{2}
\underbrace{\bigl[R_t(\bar h_w)-R_t(\bar h_t)\bigr]}
_{\text{oracle drift}}
\le
\underbrace{\mathrm{dis}_\varepsilon^{\mathrm{loc}}(P_w,P_t)}
_{\substack{\text{localized}\\\text{discrepancy}}}
\le
\underbrace{\mathrm{dis}(P_w,P_t)}_{\text{discrepancy}}.
\end{align}
Hence \eqref{eq:decomposition_localized} refines \eqref{eq:decomposition_discrepancy} whenever the localization contains both oracles.

The three quantities form a hierarchy.
The discrepancy compares risk differences over all of $\Hcal$, the localized discrepancy compares the same risk difference over predictors near the weighted oracle, and the oracle drift measures movement of the optimizer itself.
Localization does not remove changes of the risk level: since $\bar h_w\in\Hcal_{w,\varepsilon}$, the localized discrepancy is at least $|R^w(\bar h_w)-R_t(\bar h_w)|$ (Section~\ref{sec:noise-drift}).

\paragraph{Bayes Comparison}
The alternative Bayes decomposition \citep{brock2026fast} replaces the target class oracle by the weighted Bayes rule
$h_w^*\in\argmin_{h\ \mathrm{measurable}}R^w(h)$:
\begin{align*}
&R_t(\hat h_w)-R_t(h_t^*)
=\underbrace{R_t(\hat h_w)-R_t(\bar h_w)}_{\text{estimation}}\\
&\quad+\underbrace{R_t(\bar h_w)-R_t(h_w^*)}_{\text{weighted approximation}}+\underbrace{R_t(h_w^*)-R_t(h_t^*)}_{\text{Bayes drift}}.
\end{align*}
The Bayes drift depends only on the weights and the data-generating process (DGP), not on $\Hcal$.
Under square loss, with feature densities $p_i$ and $p_w=\sum_{i<t}w_ip_i$,
\[
 h_w^*(z)=\sum_{i<t}\frac{w_ip_i(z)}{p_w(z)}h_i^*(z),\qquad p_w(z)>0.
\]
The effective weights are feature-dependent under covariate drift.
Thus $h_w^*$ can be difficult to analyze and need not lie in $\Hcal$ even when the individual Bayes rules do.
This motivates focusing on the in-class oracle when studying the effect of drift on a chosen class.

The rest of the paper is organized as follows. Section~\ref{sec:estimation} gives the general results for the oracle drift and the estimation error, and Sections~\ref{sec:linear_models} and~\ref{sec:nonparametric-extension} provide the analysis for specific regression scenarios.

\section{General Results}\label{sec:estimation}
\subsection{Oracle Drift}\label{sec:deterministic}\label{sec:in-class-drift}

We control the oracle drift through a first-order comparison of the weighted and target risks with a local strong convexity-type condition.
Let $\Hcal=\{h_\theta:\theta\in\Theta\subseteq\mathbb R^p\}$, where \(R_t\) and \(R^w\) are differentiable. Denote
\[
    \theta_w\in\argmin_{\theta\in\Theta}R^w(\theta),
    \qquad
    \theta_t\in\argmin_{\theta\in\Theta}R_t(\theta),
\]
so that $\bar h_w=h_{\theta_w}$ and $\bar h_t=h_{\theta_t}$, and assume $[\theta_t,\theta_w]=\{s\theta_t + (1 - s)\theta_w\colon s\in[0, 1]\} \subseteq \Theta$.
\begin{theorem}\label{thm:parametric_oracle_stability_local}
Assume that, for some \(c_w>0\),
\begin{equation}\label{eq:curvature}
        \left\langle
    \nabla R_t(\theta_w)-\nabla R_t(\theta_t),
    \theta_w-\theta_t
    \right\rangle
    \ge
    c_w\|\theta_w-\theta_t\|_2^2 .
\end{equation}
Assume further that there exists \(\Delta_w<\infty\) such that
\[
    \|\nabla R^w(\theta_w)-\nabla R_t(\theta_w)\|_2
    \le
    \Delta_w.
\]
If  $\|h_{\theta_w}-h_{\theta_t}\|_{L_2(P_t)}
    \le
    L_w\|\theta_w-\theta_t\|_2$ for some $L_w<\infty$, then
\[
    \|\bar h_w-\bar h_t\|_{L_2(P_t)}^2
    \le
    \frac{L_w^2}{c_w^2}\Delta_w^2 .
\]
\end{theorem}
The inequality \eqref{eq:curvature} holds, e.g., if $R_t$ is strongly convex near the segment $[\theta_t,\theta_w]$.
When $R_t$ is twice differentiable, \eqref{eq:curvature} can still hold if the Hessian vanishes at isolated points, given positive average curvature along $[\theta_t,\theta_w]$.
If this curvature degenerates as the oracles approach, or vanishes in their displacement direction, then $c_w$ may shrink to zero and the quadratic bound can weaken or cease to provide control.
The parametrization only requires Lipschitzness between the two oracles.
$\Delta_w$ measures the first-order mismatch between the weighted risk and the risk at time $t$, evaluated at $\theta_w$, yielding a quadratic bound on the oracle movement.
Unlike discrepancy-based bounds,
we do not require global control of $\Theta$.
The proof is given in Appendix~\ref{app:proof:parametric_oracle_stability_local}.

\subsection{Estimation}

Our fast-rate oracle inequality for the estimation error is closely related to \citet{brock2026fast}, who bound the learning error $R_t(\hat h_w)-R_t(h_w^*)$.
However, we separate the approximation error from the estimation error in decomposition \eqref{eq:decomposition_class_oracle_2}, making the learning result purely class-dependent.

The full conditions are in Appendix~\ref{app:assumptions}.
Essentially, we require transferability of bounds across time, control of local deviations in the empirical risk, and a Bernstein condition that is commonly imposed for fast rates.
Write $\Pcal=\{P_i:1\le i\le t\}\cup\{P_w:w\in\Wcal\}$, $\|f\|_{L_2(P)}=(\E_P[f^2])^{1/2}$, and $\|f\|_{L_2(\Pcal)}=\sup_{P\in\Pcal}\|f\|_{L_2(P)}$.

\begin{theorem}
\label{thm:cr-fast}
Fix \(\delta\in(0,1)\), and assume
\eqref{cond:cr-loss-upper}--\eqref{cond:cr-simplify}. Let \(\Wcal\) be a class
of weight vectors.
Then, with probability at least \(1-2\delta\), it holds for all \(w\in\Wcal\):
\begin{align*}
    \|\hat h_w-\bar h_w\|_{L_2(\Pcal)}^2
    \lesssim
    r(\|w\|_2)^2\log^2\left(1/\delta\right).
\end{align*}
\end{theorem}

The full theorem statement and proof are given in Appendix~\ref{app:cr-fast-full}.
The rate function $r(\|w\|_{2})^2$ contains the complexity of $\Hcal,\Wcal$, the dependence of the data, and the \textit{effective sample size} $n_{\mathrm{eff}}=1/\|w\|_2^2$ \citep{kish1992weighting}.
Naturally, this leads to a bias--variance trade-off: using more data increases $n_{\mathrm{eff}}$, whereas emphasizing recent observations reduces temporal bias.
While the rate is generally sharper for larger $n_{\mathrm{eff}}$, it must be carefully balanced with drift and approximation (see Figure~\ref{fig:smooth-bias-variance-tradeoff}).

Concrete estimation rates for linear models and basis expansions are combined with the drift results to obtain explicit excess-risk bounds in
Sections~\ref{sec:linear_models} and~\ref{sec:nonparametric-extension}.

\section{Linear Regression}\label{sec:linear_models}
We now apply the previous results to linear regression with square loss. We compare oracle drift with the (localized) discrepancy in Sections~\ref{sec:curature-linear} and~\ref{sec:discrepancy-linear}, respectively.
Section~\ref{sec:drift-models} applies the theory to concrete drift scenarios and combines it with estimation to give full excess-risk bounds. The proofs and further technical details are stated in Appendix~\ref{app:linear-nonparametric}.

\paragraph{Setup} Let $X_t = (Y_t, Z_t) \in [-B, B] \times \mathbb R^p$, $t \ge 1$, and $L(X_t, h) = (Y_t - h(Z_t))^2$.
This includes autoregressive problems in which $Z_t$ contains lagged values of $Y_t$.
The optimal hypothesis under $P_t$ is the conditional expectation $h_t^*(z)=\E[Y_t\mid Z_t=z]$, and $R_t(h)-R_t(h_t^*)=\|h-h_t^*\|_{L_2(P_t)}^2$ for every $h$.
Using this identity and $(a+b)^2 \le 2(a^2+b^2)$,
\begin{align*}
    R_t(\bar h_w)-R_t(h^*_t)\!\lesssim\! \| \bar h_w - \bar h_t\|^2_{L_2(P_t)}\!+\! \| \bar h_t - h^*_t\|^2_{L_2(P_t)},
\end{align*}
allowing us to apply Theorem~\ref{thm:parametric_oracle_stability_local} to control the drift.
Let \(Z_t\in\Zcal\subset\mathbb R^p\), with
\(\sup_{z\in\Zcal}\|z\|_2\le1\), and consider
\[
    \Hcal=\{z\mapsto \beta^\top z:\|\beta\|_2\le B\},
\]
where we have $\Theta=\{\beta:\|\beta\|_2\le B\}$. The data are generated according to
\[
    Y_t=\beta_t^{*\top}Z_t+\eta_t,
    \qquad
    \|\beta_t^*\|_2\le B,
    \qquad
    \E[\eta_t\mid Z_t]=0 .
\]
For \(i\le t\), define the second moment, cross moment, and Gram matrix by
\begin{align*}
    a_i=\E_{P_i}[Y^2],
    \quad
    b_i=\E_{P_i}[ZY], \quad
    G_i=\E_{P_i}[ZZ^\top],
\end{align*}
and let
\begin{align*}
    a_w=\sum_{i=1}^{t-1}w_i a_i,
    \quad
    b_w=\sum_{i=1}^{t-1}w_i b_i,
    \quad
    G_w=\sum_{i=1}^{t-1}w_i G_i .
\end{align*}
The weighted and target class oracles are
\[
    \beta_w\in\argmin_{\|\beta\|_2\le B}R^w(\beta),
    \qquad
    \beta_t\in\argmin_{\|\beta\|_2\le B}R_t(\beta),
\]
with \(\bar h_w(z)=\beta_w^\top z\) and
\(\bar h_t(z)=\beta_t^\top z\). Throughout this section we assume the target
Gram matrix is well conditioned:
\[
    \gamma I_p\preceq G_t\preceq \Gamma I_p
    \qquad
    \text{for some }0<\gamma\le\Gamma<\infty .
\]

\subsection{Oracle Drift}\label{sec:curature-linear}
Theorem~\ref{thm:parametric_oracle_stability_local} controls the movement of the in-class oracle through the
first-order optimality equations. We have
\[
    \nabla R_i(\beta)=2(G_i\beta-b_i),
    \qquad
    \nabla R^w(\beta)=2(G_w\beta-b_w).
\]
We write
\[
    v_w = (G_w-G_t)\beta_w-(b_w-b_t).
\]
This residual measures how far \(\beta_w\) is from satisfying the normal
equation at time $t$. In particular, it evaluates the effect of the distributional change at
\(\beta_w\).
\begin{proposition}\label{prop:localnormalcontrol}
The oracle drift satisfies
\[
    \|\bar h_w-\bar h_t\|_{L_2(P_t)}^2
    \le
    \frac{\Gamma}{\gamma^2}
    \|v_w\|_2^2.
\]
\end{proposition}

Proposition~\ref{prop:localnormalcontrol} enables sharper control of the drift because it allows cancellations between
\((G_w-G_t)\beta_w\) and \(b_w-b_t\). Thus the bound on the oracle drift is solely controlled by the error in the
 first-order condition, rather than by all changes in the distribution.
This quadratic scaling of the bound is the natural excess-risk scale:
the risk is locally quadratic around its minimizer, and we can show that
this scaling is also optimal.
If the oracles lie in the interior of \(\|\beta\|_2\le B\), the normal equations
\[
    G_w\beta_w=b_w,
    \qquad
    G_t\beta_t=b_t
\]
give $v_w=G_t(\beta_t-\beta_w)$ and
\[
 R_t(\bar h_w)-R_t(\bar h_t)=\|\bar h_w-\bar h_t\|_{L_2(P_t)}^2
 =v_w^\top G_t^{-1}v_w.
\]
Thus the actual oracle drift is quadratic in the first-order residual.
\begin{proposition}\label{thm:driftoptimality}
Assume that \(\beta_w\) and \(\beta_t\) are interior points of
\(\Theta\). Then
\[
    \frac{1}{\Gamma}\|v_w\|_2^2
    \le
    \|\bar h_w-\bar h_t\|_{L_2(P_t)}^2
    \le
    \frac{1}{\gamma}\|v_w\|_2^2 .
\]
\end{proposition}
Thus, the quadratic scaling is the exact order of the oracle movement, up to constants.

\subsection{Global and Localized Discrepancy}\label{sec:discrepancy-linear}
Bounding the discrepancy requires control over the entire hypothesis space. For the square loss,
\[
    R^w(\beta)-R_t(\beta)
    =
    a_w-a_t
    -
    2\beta^\top(b_w-b_t)
    +
    \beta^\top(G_w-G_t)\beta,
\]
which yields
\[
    \mathrm{dis}(P_w,P_t)
    =
    \sup_{\|\beta\|_2\le B}|R^w(\beta)-R_t(\beta)|.
\]
For a matrix $A$, let $\|A\|_{\mathrm{op}}=\sup_{\|u\|_2=1}\|Au\|_2$.

\begin{proposition}\label{prop:discrepancy-asymp}
Assume $\|\beta_w\|_2\le B-\rho$ for some $\rho\in(0,B)$. Then
\[
    \mathrm{dis}(P_w,P_t)
    \asymp
    |a_w-a_t|
    +
    \|v_w\|_2
    +
    \|G_w-G_t\|_{\mathrm{op}},
\]
where the constants depend only on $B$ and $\rho$.
\end{proposition}

The term \(\|v_w\|_2\) is the same local perturbation that
appears in Proposition~\ref{prop:localnormalcontrol}. The other terms are global moment changes:
\(|a_w-a_t|\) is a change in the second moment of the response, while
\(\|G_w-G_t\|_{\mathrm{op}}\) captures changes in the feature covariance.

The discrepancy bound is useful if one can
control \(|a_i-a_t|\), \(\|b_i-b_t\|_2\), and
\(\|G_i-G_t\|_{\mathrm{op}}\) uniformly over time. This requires control over the entire function class,
and the resulting drift contribution is first order in $\|v_w\|_2$ compared to the oracle drift's quadratic scaling.
Further, the discrepancy can be large due to changes in \(a_w\) or \(G_w\) that do not move the oracle.
In contrast, Proposition~\ref{prop:localnormalcontrol} gives a second-order contribution when the local perturbation \(v_w\) is small and is not affected by changes in $a_i$.

\paragraph{Localized Discrepancy}
Suppose that $\beta_w$ is an interior oracle and $G_w\succ0$.
Writing $\beta=\beta_w+\tau$ and $\ell_w=R^w(\beta_w)-R_t(\beta_w)$, we have
\begin{equation*}
 R^w(\beta)-R_t(\beta)
 =\ell_w+2\tau^\top v_w+\tau^\top(G_w-G_t)\tau,
\end{equation*}
and $R^w(\beta)-R^w(\beta_w)=\tau^\top G_w\tau$, so that $\Hcal_{w,\varepsilon}$ corresponds to the feasible part of the ellipsoid $\tau^\top G_w\tau\le\varepsilon$.
The three terms are a shift of the risk level at the weighted oracle, a tilt by the first-order residual, and a change of curvature.

\begin{proposition}\label{prop:linear-localized-discrepancy}
Suppose in addition that $\{\beta_w+\tau:\tau^\top G_w\tau\le\varepsilon\}\subseteq\Theta$. Then
\begin{align*}
\mathrm{dis}_\varepsilon^{\mathrm{loc}}(P_w,P_t)
&\asymp |\ell_w|+\sqrt{\varepsilon}\|G_w^{-1/2}v_w\|_2\\
&\quad+\varepsilon\,
\|G_w^{-1/2}(G_w-G_t)G_w^{-1/2}\|_{\mathrm{op}},
\end{align*}
with absolute constants.
\end{proposition}
Compared with the discrepancy bound, the class radius $B$ is replaced by the local radius $\sqrt{\varepsilon}$ in the $G_w$-norm.
The response moment is not removed, but enters through the risk-level shift $\ell_w$.
The middle term is first order in the residual $v_w$ that generates the quadratic oracle drift, and the last term reacts to changes of $G_w$ whether or not the oracle moves.
For $\Theta=B\mathbb B_2^p$, the inclusion is implied by $\varepsilon\le\lambda_{\min}(G_w)(B-\|\beta_w\|_2)^2$.

\begin{figure*}[t]
    \centering
    \includegraphics[width=0.95\textwidth]{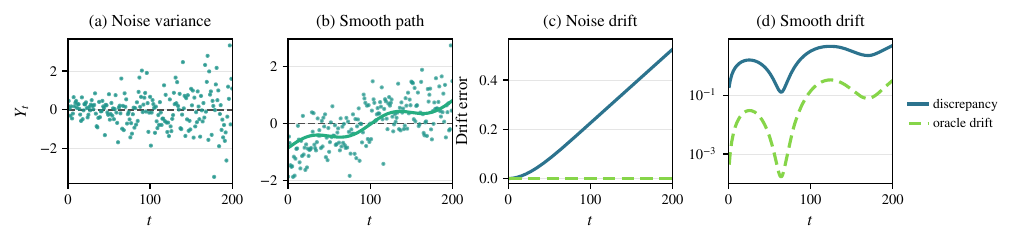}
\caption{Noise and smooth drift: discrepancy responds to irrelevant changes and exceeds oracle drift.}
    \label{fig:drift-error-illustrations}
\end{figure*}

\subsection{Excess Risk in Drift Scenarios}\label{sec:drift-models}\label{sec:excess-risk-linear}
We now apply the previous results to concrete drift settings and combine them with the estimation bound via \eqref{eq:decomposition_class_oracle_2}.
Throughout, $\bar h_t=h_t^*$ and logarithmic factors are omitted; we again refer to Appendix~\ref{app:linear-nonparametric} for technical details.
\subsubsection{Noise Drift}\label{sec:noise-drift}
Assume that the feature distribution is stationary,
\(Z_i\sim P_Z\), and that the regression coefficient is constant,
\(\beta_i^*=\beta^*\). The only nonstationarity is in the noise variance
\[
    \sigma_i^2=\E[\eta_i^2].
\]
\emph{Oracle drift.} In this setting, it holds for all $i<t$,
\begin{align*}
    G_i&=\E[ZZ^\top]=G_t,
    \qquad
    b_i=\E[ZY_i]=G_t\beta^*=b_t.
\end{align*}
Consequently,
\[
    v_w=0, \qquad  \|\bar h_w-\bar h_t\|_{L_2(P_t)}^2=0.
\]
The oracle is constant, so using older observations introduces no temporal bias. Uniform weights $w_i=n^{-1}$, with sample size $n=t-1$, use the full history and, with high probability,
\[
 R_t(\hat h_w)-R_t(h_t^*)\lesssim\frac{p}{n},
\]
which is the usual iid rate because nonstationarity of the noise does not affect the learning problem.

\emph{Discrepancy.} The discrepancy does not vanish. We have
\[
    a_i=\E[Y_i^2]=\beta^{*\top}G_t\beta^*+\sigma_i^2.
\]
\begin{proposition}\label{prop:noise-drift}
It holds
\[
    \mathrm{dis}(P_w,P_t)
    =
    \left|a_w-a_t\right|
    =
    \left|
        \sum_{i=1}^{t-1}w_i\sigma_i^2-\sigma_t^2
    \right|.
\]
\end{proposition}
This example shows that the discrepancy is nonzero whenever the weighted average of past noise
variances differs from the noise variance at time $t$, although the conditional mean and therefore the optimal predictor remain unchanged.

The reason is that the same constant is added to the risk of every predictor,
\[
 R_i(\beta)=\sigma_i^2+(\beta-\beta^*)^\top G_t(\beta-\beta^*),
\]
yielding a vertical shift in the risk surface, while its shape, and hence its minimizer, is unchanged.
The oracle drift compares the risks of two predictors, so it is invariant to such shifts.

The localized discrepancy detects the shift as well. Since $R^w-R_t$ is constant across $h$, $\mathrm{dis}_\varepsilon^{\mathrm{loc}}(P_w,P_t)=\mathrm{dis}(P_w,P_t)$ for every $\varepsilon\ge0$, so localization alone does not remove this effect.
The shift cancels for a discrepancy defined on risks centered at their respective oracles, that is, by comparing $R^w(h)-R^w(\bar h_w)$ with $R_t(h)-R_t(\bar h_t)$; see, e.g., \citet{hanneke2019value}.
 Here, both quantities equal $(\beta-\beta^*)^\top G_t(\beta-\beta^*)$.
 However, centering does not remove changes in the risk geometry as we will see in Section~\ref{sec:feature-moment-drift}.

Figure~\ref{fig:drift-error-illustrations} illustrates this with the
DGP \(Y_t=\eta_t\), \(\E\eta_t=0\), and increasing variance
\(\sigma_t^2=0.12+2.6(t/200)^2\), combined with exponential smoothing weights
\begin{align*}
    w_{i}^{(t)} = \mathbbm{1}\{1\le i<t\}\frac{\exp[-\alpha(t-i)]}{\sum_{j=1}^{t-1}\exp[-\alpha(t-j)]}, \alpha \in (0, \infty).
\end{align*}
The discrepancy grows while the oracle
drift remains zero. Details are given in Appendix~\ref{app:simulation-setup}.
By Proposition~\ref{prop:noise-drift}, we obtain the rate
\[
 R_t(\hat h_w)-R_t(h_t^*)\lesssim\frac{p}{n}+\mathrm{dis}(P_w,P_t).
\]
If the noise variance changes by a constant amount over the history, the second term is $O(1)$ and does not vanish as $t$ grows, so the bound suggests discarding old observations, yielding a slower rate.
The same holds for the localized discrepancy.

\subsubsection{Feature Drift}\label{sec:feature-moment-drift}
We now consider a shift in the feature distribution, which changes the shape of the risk surface in contrast to the previous setting.

Assume $\E[\eta_i^2]=\sigma^2$ and $\|\beta^*\|_2<B$, while $G_i$ may change over time.

\emph{Oracle drift.} The risk is
\begin{align*}
    R_i(\beta)=\sigma^2+(\beta-\beta^*)^\top G_i(\beta-\beta^*),
\end{align*}
so if $G_w\succ0$, both oracles equal $\beta^*$ and the oracle drift is zero.
The oracle is again constant, and uniform weights give the same rate $p/n$.

\emph{Discrepancy.} The discrepancies do not vanish:
\begin{proposition}\label{prop:feature-moment-drift}
Let $G_w\succ0$ and $\varepsilon\le\lambda_{\min}(G_w)(B-\|\beta^*\|_2)^2$. Then
\begin{equation}\label{eq:geometry-localized-discrepancy}
\mathrm{dis}_\varepsilon^{\mathrm{loc}}(P_w,P_t)
=\varepsilon\|G_w^{-1/2}(G_w-G_t)G_w^{-1/2}\|_{\mathrm{op}}.
\end{equation}
\end{proposition}
Hence
\begin{align*}
    0=R_t(\bar h_w)-R_t(\bar h_t)<\mathrm{dis}_\varepsilon^{\mathrm{loc}}(P_w,P_t)\le\mathrm{dis}(P_w,P_t)
\end{align*}
whenever $G_w\ne G_t$ and $\varepsilon>0$.
Localization replaces the class radius by $\varepsilon$, but cannot remove the dependence on $G_w-G_t$. Unlike in the previous example, centering does not help either: the excess risks $(\beta-\beta^*)^\top G_w(\beta-\beta^*)$ and $(\beta-\beta^*)^\top G_t(\beta-\beta^*)$ still differ through the feature geometry.
The discrepancy guarantee is instead
\[
 R_t(\hat h_w)-R_t(h_t^*)\lesssim\frac{p}{n}+\|G_w-G_t\|_{\mathrm{op}},
\]
and the localized guarantee replaces the second term by \eqref{eq:geometry-localized-discrepancy}.
The discrepancy term is $O(1)$ if the feature covariance changes by a constant amount. The localized term carries a factor $\varepsilon$, but remains affected by the feature geometry, although the oracle does not change.
Thus noise variance changes the risk level while feature moments change the risk shape. However, neither moves the oracle, which is correctly reflected by the oracle drift being zero. On the other hand, the discrepancies are nonzero in both cases.

\subsubsection{Smooth Drift}\label{sec:smooth-drift}
Assume now that both the feature distribution $Z_i\sim P_{Z,i}$ and the coefficient $\beta_i^*$ may change over time, with $Y_i=\beta_i^{*\top}Z_i+\eta_i$ and $\E[\eta_i\mid Z_i]=0$.
Let $\psi_i(t)=|t-i|^\nu$ and denote
\[
    M_w(\psi;t)=\sum_{i=1}^{t-1}|w_i||t-i|^\nu .
\]
$M_w(\psi;t)$ captures the temporal bias induced by the weights. We compare the oracle drift with the discrepancy-based drift under smoothly drifting distributions.

\emph{Oracle drift.} We impose the following local condition.

\begin{proposition}\label{prop:smooth-local}
Assume that, for some \(L_v<\infty\) and
$\nu > 0$, it holds that
\[
    \|(G_i-G_t)\beta_w-(b_i-b_t)\|_2
    \le
    L_v|t-i|^\nu
    \qquad
    \forall i<t.
\]
Then
\[
    \|\bar h_w-\bar h_t\|_{L_2(P_t)}^2
    \le
    \frac{\Gamma}{\gamma^2}
    L_v^2
    M_w(\psi;t)^2 .
\]
\end{proposition}

This assumption controls
the normal-equation drift only at the weighted oracle. It can therefore be small
even when \(G_i\) and \(b_i\) drift individually, as long as their effect on the
oracle cancels.
A sufficient condition is smooth drift of the feature covariance and of the coefficient path: if $\|G_i-G_t\|_{\mathrm{op}}\le L_G\psi_i(t)$ and $\|\beta_i^*-\beta_t^*\|_2\le L_\beta\psi_i(t)$ for all $i<t$, the local condition holds with $L_v=2BL_G+\Gamma L_\beta$.
For exponential weights, the effective horizon $s_\alpha=\min\{(1-e^{-\alpha})^{-1},\,n\}$ determines the effective sample size: the estimation error is of order $p/s_\alpha$, while the observations used are on average of age $s_\alpha$, so the temporal bias is $M_w(\psi;t)\lesssim s_\alpha^\nu$.
Suppose the drift constants are of order $\kappa$. The oracle drift then gives
\[
 R_t(\hat h_{w(\alpha)})-R_t(h_t^*)\lesssim\frac{p}{s_\alpha}+\kappa^2s_\alpha^{2\nu}.
\]

\emph{Discrepancy.} We now control the discrepancy drift.

\begin{proposition}\label{prop:smooth-discrepancy}
Assume that, for constants \(L_a,L_b,L_G<\infty\), and for all $i<t$,
\begin{align*}
    &|a_i-a_t|\le L_a\psi_i(t),
    \qquad
    \|b_i-b_t\|_2\le L_b\psi_i(t),
    \\
    &\|G_i-G_t\|_{\mathrm{op}}\le L_G\psi_i(t).
\end{align*}
Then
\[
    R_t(\bar h_w)-R_t(\bar h_t)
    \le
    2(L_a+B^2L_G+2BL_b)M_w(\psi;t).
\]
\end{proposition}

This condition is more conservative because it requires smoothness of all moment
components, while only yielding a first-order dependence on \(M_w(\psi;t)\), with
\[
 R_t(\hat h_{w(\alpha)})-R_t(h_t^*)\lesssim\frac{p}{s_\alpha}+\kappa s_\alpha^{\nu}.
\]
The bounds differ only in how the temporal bias enters: the oracle drift is quadratic, as in a classical bias--variance decomposition, whereas the discrepancy is the bias itself.
For the localized discrepancy, \eqref{eq:decomposition_localized} additionally requires $\bar h_t\in\Hcal_{w,\varepsilon}$, and the risk-level shift $|\ell_w|$ in Proposition~\ref{prop:linear-localized-discrepancy} is first order and does not shrink with $\varepsilon$, so the same bound applies.

Figure~\ref{fig:drift-error-illustrations} illustrates this contrast for the
scalar intercept model \(Y_t=\beta_t^*+\eta_t\), with a smooth seasonal
coefficient path and exponential weights on the past observations. Panel (d)
compares the discrepancy with the oracle drift.  The
exact simulation setup is given in
Appendix~\ref{app:simulation-setup}.

\emph{Slow drift.} For small $\kappa$, balancing estimation and drift gives the horizons
\[
 s_\alpha^*\asymp(p/\kappa^2)^{1/(2\nu+1)}
 \quad\text{and}\quad
 \tilde s_\alpha^*\asymp(p/\kappa)^{1/(\nu+1)},
\]
with rates $p^{2\nu/(2\nu+1)}\kappa^{2/(2\nu+1)}$ and $p^{\nu/(\nu+1)}\kappa^{1/(\nu+1)}$, respectively.
The oracle drift therefore permits a longer horizon and gives a faster rate.
Intuitively, a small drift has an even smaller squared effect on the excess risk; a first-order bound overstates it and prescribes weights that discard informative observations.

\emph{Fast drift.} For large $\kappa$, the oracle can move by more than a constant within the horizon, $\kappa s_\alpha^\nu>1$, and then the squared term becomes dominant.
Since $R_t(\bar h_w)-R_t(\bar h_t)\le2\,\mathrm{dis}(P_w,P_t)$ by \eqref{eq:decomposition_discrepancy}, both bounds apply to the same drift term, so
\[
 R_t(\hat h_{w(\alpha)})-R_t(h_t^*)\lesssim\frac{p}{s_\alpha}+\min\{\kappa^2s_\alpha^{2\nu},\,\kappa s_\alpha^{\nu}\}.
\]
The minimum is the squared term for $\kappa s_\alpha^\nu\le1$ and the first-order term otherwise, so the combined guarantee is never worse than the discrepancy guarantee.
Moreover, at the horizon minimizing $p/s_\alpha+\kappa^2s_\alpha^{2\nu}$, balancing gives $\kappa(s_\alpha^*)^\nu=(p/s_\alpha^*)^{1/2}\asymp(\kappa p^{\nu})^{1/(2\nu+1)}$, which is smaller than one whenever the guarantee is nontrivial.
The reversed regime thus only arises if $\kappa^2\gtrsim p^{-2\nu}$, where every horizon gives a bound of order one and the minimum is attained by the first-order discrepancy term.

\paragraph{Bias--Variance Trade-off}
\begin{figure}[t]
    \centering
    \includegraphics[width=0.78\linewidth]{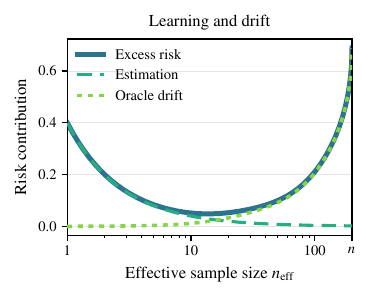}
    \caption{Bias--variance trade-off for exponential weights in the smooth-drift
    setting.}
    \label{fig:smooth-bias-variance-tradeoff}
\end{figure}
Figure~\ref{fig:smooth-bias-variance-tradeoff} shows the
oracle drift, estimation error, and excess risk of the weighted
estimator under a smoothly drifting oracle. Nearly uniform weights attain the maximal
effective sample size \(n_{\mathrm{eff}}=n\) and minimize estimation error,
but retain the longest history and can therefore induce substantial temporal
bias. Concentrating on recent observations reverses these effects
(\(n_{\mathrm{eff}}=1\)). Thus neither extreme is automatically optimal: the
complete excess risk balances statistical error against drift.

\section{Nonparametric Regression}
\label{sec:nonparametric-extension}

Unlike the fixed linear class, nonparametric regression allows the approximation space to grow with the data. Consider $Y_i=h_i^*(Z_i)+\eta_i$ with $\E[\eta_i\mid Z_i]=0$.
Let $\Phi_q:\mathcal Z\to\mathbb R^q$ be a feature map and consider the
finite-dimensional approximation class
\[
    \Hcal_q
    =
    \left\{
        h_\beta(z)=\beta^\top\Phi_q(z):
        \beta\in\Theta_q
    \right\},
\]
where $\Theta_q\subset\mathbb R^q$ is convex. For each fixed $q$, the
analysis of Section~\ref{sec:linear_models} applies after replacing
$Z$ by $\Phi_q(Z)$. The number of basis functions $q$ now affects both approximation and oracle drift.
For equal-width $q$-bin histograms on $[0,1)$, assume $\|h_i^*\|_\infty\le B$ and take $\Theta_q=\Theta_q^{\mathrm{hist}}=[-B,B]^q$.
Define
\[
    G_{i,q}
    =
    \E[\Phi_q(Z_i)\Phi_q(Z_i)^\top],
    \qquad
    b_{i,q}
    =
    \E[\Phi_q(Z_i)Y_i],
\]
and let $G_{w,q}$ and $b_{w,q}$ denote their weighted versions. Write
$\beta_{w,q}$ and $\beta_{t,q}$ for the weighted and in-class
oracles. The resulting local perturbation is
\[
    v_{w,q}
    =
    (G_{w,q}-G_{t,q})\beta_{w,q}
    -(b_{w,q}-b_{t,q}).
\]

\paragraph{Drift for Fixed $q$}
Suppose the class oracles are interior and $\gamma_q I_q\preceq G_{t,q}\preceq\Gamma_q I_q$ for $0<\gamma_q\leq\Gamma_q<\infty$. By the same argument as in
Proposition~\ref{thm:driftoptimality}, we have
\[
    \frac{1}{\Gamma_q}\|v_{w,q}\|_2^2
    \leq
    \|\bar h_{w,q}-\bar h_{t,q}\|_{L_2(P_t)}^2
    \leq
    \frac{1}{\gamma_q}\|v_{w,q}\|_2^2.
\]
As in linear regression, only the combined change in $G_{w,q}$ and $b_{w,q}$ moves the oracle, and its risk contribution is quadratic. The choice of $q$ also changes the projection and conditioning, so it can affect the size of this drift.
Moreover, the discrepancy can remain first order and react to noise-variance changes even when the oracle is fixed (see Appendix~\ref{app:nonparametric} for details).

\paragraph{Weights under Smooth Drift}
Denote the approximation term by
$A_t(q)=R_t(\bar h_{t,q})-R_t(h_t^*)$. Under exponential weights, we obtain
\begin{equation*}
 R_t(\hat h_{w(\alpha),q})-R_t(h_t^*)
 \lesssim \frac{q}{s_\alpha}+D_qs_\alpha^{2\nu}+A_t(q),
\end{equation*}
where $D_q$ captures how the drift of the $q$-dimensional in-class oracle
depends on the chosen basis resolution.
Thus weights and resolution must balance all three terms. Unlike in the
fixed-dimensional linear model, the resolution $q$ is itself a tuning
parameter: increasing $q$ reduces approximation error but increases
estimation error and can also change the projected drift through $D_q$.
Hence the temporal horizon and model resolution should generally be chosen
jointly. For one-dimensional Lipschitz histograms,
$A_t(q)\lesssim q^{-2}$ and $q\asymp s_\alpha^{1/3}$ balances estimation and
approximation for a fixed horizon. This conditional choice may not be
optimal once drift is included: if $D_q$ grows with $q$, the drift term
favors a coarser basis and shifts the joint optimum.
Figure~\ref{fig:basis-resolution-tradeoff} illustrates this for histograms under a smoothly drifting regression function: estimation grows with $q$, approximation decays, and the oracle drift increases as finer bins resolve more of the drift. Details are given in Appendix~\ref{app:simulation-setup}.
\begin{figure}[t]
    \centering
    \includegraphics[width=0.78\linewidth]{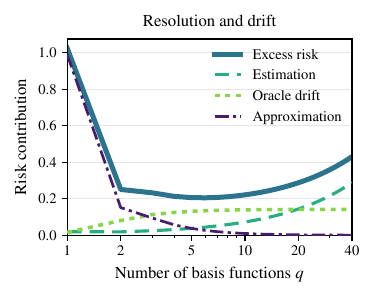}
    \caption{Resolution trade-off for histogram bases with exponential weights in the smooth-drift setting.}
    \label{fig:basis-resolution-tradeoff}
\end{figure}

\section{Discussion}\label{sec:discussion}

Our results give a direct way to control oracle drift in parametrized classes under local conditions around the two oracles.
The comparison is pointwise, requires no supremum over the function class, and does not rely on global geometric conditions on the parameter space.
Combined with our estimation results, this yields explicit excess-risk guarantees that separate statistical error from the effect of distributional drift.

For linear regression, the resulting drift contribution has the optimal quadratic scaling, while discrepancies can remain first order and can react to changes that leave the oracle unchanged.
This distinction persists even when the risks are centered at their oracles, because centering does not remove changes in the risk geometry.
The same perspective extends to growing approximation classes in nonparametric regression, where estimation, oracle drift, and approximation must be balanced jointly.

An important but orthogonal question to our drift analysis is data-driven selection of the weights.
This requires additional information or assumptions about the target; discrepancy-based algorithms likewise rely on conditions such as a recent window being close to the future \citep{KuznetsovMohri2020} or labeled target data \citep{pmlr-v206-awasthi23b}, while related work develops adaptive weighting or look-back selection under additional structural assumptions \citep{huangWang2025stability,baby2025adaptive}.
Such conditions make data-driven selection possible, but this is separate from whether the drift quantity being used is closely tied to the movement of the prediction oracle.
Our analysis addresses the latter question for a given weighting scheme.
Adaptive procedures targeted at the oracle drift and more general drift models, such as structural breaks, are natural directions for future research.

\label{main-text-end}

\subsection*{AI use statement}
In this work, we used generative AI tools to assist with deriving and checking proofs of theoretical results and with implementing the simulations.
The proof assistance was used in particular for the results on linear and nonparametric regression.
Additionally, we used generative AI tools for standard editing and writing tasks.
We have reviewed all AI-assisted work: all AI-assisted proofs were checked line by line by the authors, and the simulation code was verified by the authors.
We take responsibility for the final content of this work, including text, claims, or artifacts produced with the aid of generative AI.

\bibliographystyle{apalike}
\bibliography{bibliography}

\clearpage
\appendix
\thispagestyle{empty}

\onecolumn
\aistatstitle{Beyond Discrepancy: Oracle Drift in Nonstationary Weighted Risk Minimization\\
Supplementary Materials}

\section{Extended Related Work}\label{sec:related_work}

Learning under nonstationarity is often analyzed through discrepancy-type quantities that compare expected losses under past and target distributions over a hypothesis class.
This perspective is closely related to discrepancy in domain adaptation \citep{mansour2009domain} and has been used to derive guarantees for drifting or nonstationary processes \citep{MohriMunozMedina2012,KuznetsovMohri2016,KuznetsovMohri2017,KuznetsovMohri2020,HannekeYang2019,pmlr-v206-awasthi23b,mazzetto2023adaptivealgorithmlearningunknown}.
These analyses concern different discrepancies and comparators; here we focus on the target excess risk of the weighted in-class oracle.
Discrepancy-type quantities appear both for marginal risks and for conditional expected losses given the past, and generalization bounds may only require a one-sided comparison of the two risks.
The add-and-subtract argument behind \eqref{eq:decomposition_discrepancy} produces risk differences in opposite directions, so the symmetric version with an absolute value appears naturally in our setting.

Localization enters this literature through several mechanisms.
The risk discrepancy can be restricted to a subset of the class, as in the localized discrepancy discussed in Section~\ref{sec:decompositions}.
Alternatively, excess risks relative to the respective minimizers can be compared, which removes additive shifts of the risk surface; such comparisons appear in transfer learning, e.g., \citet{hanneke2019value} show limitations of localizing to small source excess risk.
In domain adaptation, \citet{zhang2020localized} define localized $\Hcal\Delta\Hcal$ disagreement and disparity discrepancies for binary classification, and \citet{demathelin2022discrepancy} localize a class of Lipschitz predictors by agreement with observed labels, for losses satisfying a triangle inequality.
Discrepancies are also studied in regression: \citet{mansour2009domain} consider general losses, and \citet{cortes2011regression} give regression guarantees, including results for squared loss.

Data-dependent discrepancy procedures require additional information or stability assumptions, such as a recent window that is representative of the unseen future \citep{KuznetsovMohri2020} or labeled target data \citep{pmlr-v206-awasthi23b}; weaker supervision is possible under additional compatibility assumptions.

Weighted ERM provides a natural way to adapt to temporal drift by trading off variance against bias.
Several recent works study this trade-off through adaptive weighting, look-back selection, or temporally robust model selection \citep{mazzetto2023adaptivealgorithmlearningunknown,huangWang2025stability,temporalshifts,baby2025adaptive,keehan2025wasserstein}.
Our work is complementary. Rather than focusing on the selection of a specific weighting rule, we study how the excess risk should be decomposed for weighted risk minimization.
Our results hold uniformly over a weight class $\Wcal$, but do not focus on the specific choice of the weights.

The estimation part of our analysis is related to fast-rate learning theory for dependent data.
Oracle inequalities and fast rates under mixing assumptions have been established for stationary or weakly dependent processes \citep{SteinwartChristmann2009,HangSteinwart2014,HangSteinwart2017,FarahmandSzepesvari2012,BarreraGobet2021,AlquierLiWintenberger2013,DebMukherjee2024}.
More recently, \citet{brock2026fast} prove fast-rate oracle inequalities for nonstationary weighted ERM, uniformly over weight classes and with rates depending explicitly on the effective sample size.
We use this estimation analysis to combine in-class oracle drift with target approximation.

\section{Proof of Main Results}\label{app:cr-fast-full}

\subsection{Oracle Drift}

\label{app:proof:parametric_oracle_stability_local}

\begin{proof}[Proof of Theorem~\ref{thm:parametric_oracle_stability_local}]
Since the segment between $\theta_w$ and $\theta_t$ lies in $\Theta$ and $\theta_w,\theta_t$ are minimizers,
\[
    \langle\nabla R^w(\theta_w),\theta_t-\theta_w\rangle\ge0,
    \qquad
    \langle\nabla R_t(\theta_t),\theta_w-\theta_t\rangle\ge0.
\]
By the curvature condition,
\begin{align*}
c_w\|\theta_w - \theta_t\|_2^2
&\le
\langle
\nabla R_t(\theta_w)-\nabla R_t(\theta_t),
\theta_w-\theta_t
\rangle \\
&=
\langle
\nabla R_t(\theta_w)-\nabla R^w(\theta_w),
\theta_w-\theta_t
\rangle
+
\langle
\nabla R^w(\theta_w),
\theta_w-\theta_t
\rangle
-
\langle
\nabla R_t(\theta_t),
\theta_w-\theta_t
\rangle \\
&\leq \langle
\nabla R_t(\theta_w)-\nabla R^w(\theta_w),
\theta_w-\theta_t
\rangle \\
&\le
\|\nabla R_t(\theta_w)-\nabla R^w(\theta_w)\|_2\|\theta_w-\theta_t\|_2 \\
&\le
\Delta_w\|\theta_w-\theta_t\|_2,
\end{align*}
where the last two inequalities follow from Cauchy--Schwarz and by assumption.
Thus
\[
    \|\bar h_w-\bar h_t\|_{L_2(P_t)}
    =
    \|h_{\theta_w}-h_{\theta_t}\|_{L_2(P_t)}
    \le
    L_w\|\theta_w-\theta_t\|_2
    \le
    \frac{L_w}{c_w}\Delta_w.
\]
Squaring both sides gives the result.
\end{proof}

\subsection{Estimation Error}

\subsubsection{Mixing coefficients}\label{sec:mixing}

\begin{definition}[$\beta$-mixing]\label{def:beta}
For $k\in \mathbb{N}_0$, let $\mathcal{A}_t = \sigma\{X_j : j\le t\}$ and
$\mathcal{B}_t = \sigma\{X_j : j\ge t\}$. The $\beta$-mixing coefficient at
lag $k \ge 0$ is defined as
\begin{align*}
    \beta(k)
    =
    \sup_{t\geq1}\,
    \frac{1}{2}
    \sup_{\{A_i\}_i \in \Pi_t,\{B_j\}_j \in \Pi'_{t+k}}
    \sum_{i=1}^{I}\sum_{j=1}^{J}
    |\mathbb{P}(A_i\cap B_j)-\mathbb{P}(A_i)\mathbb{P}(B_j)|,
\end{align*}
where $\Pi_t$ and $\Pi'_{t+k}$ denote the sets of all finite partitions of
$\mathcal{A}_t$ and $\mathcal{B}_{t+k}$, respectively.
\end{definition}

\begin{definition}[$\rho$-mixing]\label{def:rho}
The $\rho$-mixing coefficient at lag $k \ge 0$ is defined as
\begin{align*}
    \rho(k)
    =
    \sup_{t \ge 1}\sup_{f,g}
    \frac{|\mathrm{Cov}(f(X_t),g(X_{t+k}))|}
    {\sqrt{\mathrm{Var}(f(X_t))\mathrm{Var}(g(X_{t+k}))}},
\end{align*}
with the inner supremum taken over all functions for which the ratio is
well-defined. Its long-run characteristic is
$K_\rho = 1 + 2\sum_{k=1}^\infty\rho(k)$.
\end{definition}

\subsubsection{Coupling and concentration for \texorpdfstring{$\beta$}{beta}-mixing processes}\label{app:coupling}
Let $X_1,\ldots,X_{n}$, with $n=t-1$, be the estimation sample, and define the weighted empirical process $\mathbb G_n^wf=\sum_{i=1}^{n}w_i\{\E f(X_i)-f(X_i)\}$. Divide the sample into alternating
blocks of size $m\in\mathbb N$, assuming without loss of generality that
$2m$ divides $n$. By maximal coupling
\citep[e.g.,][Theorem~5.1]{rio2017asymptotic}, there are blocks $U_j^*$ with
the same marginal laws as the original blocks $U_j$, such that the even and
odd sequences of coupled blocks are each independent and
\[
 \Pr(\exists j:U_j\ne U_j^*)\le \frac{n}{m}\beta(m).
\]
Write $\widetilde{\mathbb G}_n^{w,m}$ for $\mathbb G_n^w$ with the original blocks replaced by their coupled versions. Here, $N_1(\epsilon,\Wcal)$ and $N_\infty(\epsilon,\Hcal)$ denote the $\epsilon$-covering numbers of $\Wcal$ with respect to $\|\cdot\|_1$ and of $\Hcal$ with respect to $\|\cdot\|_\infty$, respectively. The following are
the fixed-class covering and concentration bounds used below; they are the
corresponding steps in the proof of Theorem~4.2 of
\citet{brock2026fast}.

\begin{lemma}\label{lem:covering}
Use the notation of Theorem~\ref{thm:cr-fast-full} below and define, for
$k\ge1$,
\begin{align*}
 H'_{k,w}&=\{(h,h')\in\Hcal^2:\|h-h'\|_{L_2(\Pcal)}
 \le C_\Pcal2^{k/2+1}r(\|w\|_2)\},\\
 \overline H'_{k,w}&=\{(h,h')\in\Hcal^2:\|h-h'\|_{L_2(\Pcal)}
 \le C_\Pcal2^{(k+3)/2}r(\|w\|_2)\}.
\end{align*}
For any random variables $X_1,\ldots,X_{n}$, it holds that
\begin{align*}
&\Pr\left(\sup_{w\in\Wcal,(h,h')\in H'_{k,w}}
r(\|w\|_2)^{-2}|\mathbb G_n^w(L_h-L_{h'})|>2^k-\frac14\right)\\
&\quad\le N_1(\epsilon_\Wcal,\Wcal)\sup_{w\in\Wcal}\Bigg\{
N_\infty(\epsilon_w,\Hcal)^2
\sup_{(h,h')\in\overline H'_{k,w}}
\Pr\left(r(\|w\|_2)^{-2}|\mathbb G_n^w(L_h-L_{h'})|>2^{k-1}\right)\Bigg\}.
\end{align*}
\end{lemma}

\begin{lemma}\label{lem:bernstein}
If $\max_{1\le i\le n}|f(X_i)|\le b$ almost surely and
$\|f\|_{L_2(\Pcal)}^2\le v$, then, for every $s>0$,
\[
 \Pr\bigl(|\widetilde{\mathbb G}_n^{w,m}f|>s\bigr)
 \le4\exp\left(-\frac{s^2}{8v\|w\|_2^2K_\rho
 +3mb\|w\|_\infty s}\right).
\]
\end{lemma}

\subsubsection{Assumptions}\label{app:assumptions}

For the weight class $\Wcal$, assume that every $P_w$ is a probability measure and define $\Pcal=\{P_1,\ldots,P_t\}\cup\{P_w:w\in\Wcal\}$.

\begin{itemize}
    \Aconditem\label{cond:cr-loss-upper}
    There is \(C_\Pcal<\infty\) such that for all \(h,h'\in\Hcal\),
    \[
        \sup_{P\in\Pcal}\|h-h'\|_{L_2(P)}
        \le
        C_\Pcal\inf_{P\in\Pcal}\|h-h'\|_{L_2(P)}.
    \]

    \Aconditem\label{cond:cr-infty}
    There is \(C_\infty\ge0\) such that for all \(h,h'\in\Hcal\),
    \[
        \inf_{P\in\Pcal}\|h-h'\|_{L_2(P)}
        \ge
        C_\infty\|h-h'\|_\infty.
    \]
    We use the convention $C_\infty^{-1}=\infty$ when $C_\infty=0$.

    \Aconditem\label{cond:cr-bernstein}
    For every \(w\in\Wcal\) and
    all \(h\in\Hcal\),
    \[
        R^w(h)-R^w(\bar h_w)
        \ge
        \|h-\bar h_w\|_{L_2(P_w)}^2.
    \]

    \Aconditem\label{cond:cr-simplify}
    The loss \(L\) is uniformly \(1\)-Lipschitz and bounded by \(1\), and all
    hypotheses \(h\in\Hcal\) are uniformly bounded by \(1\).
\end{itemize}

\begin{theorem} \label{thm:cr-fast-full}
Fix \(\delta\in(0,1)\), \(K\in(0,\infty)\), and assume
\eqref{cond:cr-loss-upper}--\eqref{cond:cr-simplify}. Let \(\Wcal\) be a class
of weight vectors and define
\begin{align*}
    C_1 &= \sup_{w\in\Wcal}\|w\|_1,
    \qquad
    C_\Wcal=\inf_{w\in\Wcal}\|w\|_{2},
    \qquad
    B_\Wcal=\sup_{w\in\Wcal}\frac{\|w\|_\infty}{\|w\|_{2}^2}, \qquad
    m_\beta = \inf\left\{m:\frac{n}{m}\beta(m)\le\delta\right\}, \\
    K_w &= 4+\log\!\left(
        N_1(\epsilon_\Wcal,\Wcal)\,
        N_\infty(\epsilon_w,\Hcal)^2
    \right), \qquad
    \epsilon_\Wcal = \min\left\{\frac{C_\Wcal^3}{64(1+C_\Wcal+C_1K)},\frac{C_\Wcal}{64(1+K)}\right\},
    \qquad
    \epsilon_w=\frac{\|w\|_{2}^2}{32C_1}.
\end{align*}
Let \(r:[C_\Wcal,C_1]\to[0,\infty)\) be increasing and \(K\)-Lipschitz, and
suppose that for all \(w\in\Wcal\),
\[
    r(\|w\|_{2})^2
    \ge
    K_w\|w\|_{2}^2
    \left(
        C_\Pcal^2K_\rho
        +
        m_\beta B_\Wcal
        \min\{2,C_\Pcal C_\infty^{-1}r(\|w\|_{2})\}
    \right).
\]
Then, with probability at least \(1-2\delta\), it holds for all \(w\in\Wcal\):
\[
    C_\Pcal^{-2}\|\hat h_w-\bar h_w\|_{L_2(\Pcal)}^2
    \lesssim
    R^w(\hat h_w)-R^w(\bar h_w)
    \lesssim
    r(\|w\|_{2})^2\log^2(1/\delta).
\]
\end{theorem}
\begin{proof}
The proof follows the same arguments as the proof of Theorem~4.2 of \citet{brock2026fast}, but with all shells and localized classes centered at $\bar h_w$ instead of $h_w^*$.
For fixed $w$, it holds $R_t^w(\hat h_w)\le R_t^w(\bar h_w)$. Hence
\begin{align*}
 R^w(\hat h_w)-R^w(\bar h_w)
 &=(R^w-R_t^w)(\hat h_w)-(R^w-R_t^w)(\bar h_w)+R_t^w(\hat h_w)-R_t^w(\bar h_w)\\
 &\le(R^w-R_t^w)(\hat h_w)-(R^w-R_t^w)(\bar h_w)\\
 &=\mathbb G_n^w(L_{\hat h_w}-L_{\bar h_w}),
\end{align*}
where $L_h=L(\cdot,h)$ and $\mathbb G_n^w(L_h)=R^w(h)-R_t^w(h)$.
We now want to bound $\Pr\{\sup_w r(\|w\|_2)^{-2}(R^w(\hat h_w)-R^w(\bar h_w))>2^M\}$.
To this end, define $\Hcal_{k,w}=\{h\in\Hcal:R^w(h)-R^w(\bar h_w)\le2^{k+1}r(\|w\|_2)^2\}$ for every $k\ge1$ and $w\in\Wcal$.
By peeling and the coupling, we have
\begin{align*}
&\Pr\Bigl\{\sup_{w\in\Wcal}r(\|w\|_2)^{-2}\bigl(R^w(\hat h_w)-R^w(\bar h_w)\bigr)>2^M\Bigr\}\\
&\quad\le\sum_{k=M}^\infty
 \Pr\left(\sup_{w\in\Wcal}\sup_{h\in\Hcal_{k,w}}
 r(\|w\|_2)^{-2}|\widetilde{\mathbb G}_n^{w,m_\beta}(L_h-L_{\bar h_w})|>2^k\right)
 +\frac{n}{m_\beta}\beta(m_\beta),
\end{align*}
where the last term is at most $\delta$ by the definition of $m_\beta$.
Further, for any $h\in\Hcal$ (in particular for $h\in\Hcal_{k,w}$), \eqref{cond:cr-loss-upper} and \eqref{cond:cr-bernstein}, using $P_w\in\Pcal$, imply
\[
 C_\Pcal^{-1}\|h-\bar h_w\|_{L_2(\Pcal)}
 \le\inf_{P\in\Pcal}\|h-\bar h_w\|_{L_2(P)}
 \le\|h-\bar h_w\|_{L_2(P_w)}
 \le\bigl[R^w(h)-R^w(\bar h_w)\bigr]^{1/2}.
\]
Therefore, we get
\[
 \|h-\bar h_w\|_{L_2(\Pcal)}
 \le C_\Pcal\bigl[R^w(h)-R^w(\bar h_w)\bigr]^{1/2}
 \le C_\Pcal2^{(k+1)/2}r(\|w\|_2)
 \le C_\Pcal2^{k/2+1}r(\|w\|_2),
\]
so that $(h,\bar h_w)\in H'_{k,w}$ with $H'_{k,w}$ and $\overline H'_{k,w}$ as in Lemma~\ref{lem:covering}.
For every $(h,h')\in\overline H'_{k,w}$, \eqref{cond:cr-infty} and \eqref{cond:cr-simplify} further give
\[
 \|L_h-L_{h'}\|_\infty\le\min\{2,C_\Pcal C_\infty^{-1}2^{(k+3)/2}r(\|w\|_2)\},
 \qquad
 \|L_h-L_{h'}\|_{L_2(\Pcal)}^2\le C_\Pcal^22^{k+3}r(\|w\|_2)^2.
\]
Hence, by Lemma~\ref{lem:covering} applied to the coupled process and Lemma~\ref{lem:bernstein} with $s=2^{k-1}r(\|w\|_2)^2$, and the rate condition, it holds for $M=\lceil2\log_2(96+24\log(34/\delta))\rceil$,
\begin{align*}
&\Pr\Bigl\{\sup_{w\in\Wcal}r(\|w\|_2)^{-2}\bigl(R^w(\hat h_w)-R^w(\bar h_w)\bigr)>2^M\Bigr\}\\
&\quad\le\sum_{k=M}^\infty
 \Pr\left(\sup_{w\in\Wcal,(h,h')\in H'_{k,w}}
 r(\|w\|_2)^{-2}|\widetilde{\mathbb G}_n^{w,m_\beta}(L_h-L_{h'})|>2^k\right)+\delta\\
&\quad\le\sum_{k=M}^\infty\sup_{w\in\Wcal}e^{K_w}
 \sup_{(h,h')\in\overline H'_{k,w}}
 \Pr\left(r(\|w\|_2)^{-2}|\widetilde{\mathbb G}_n^{w,m_\beta}(L_h-L_{h'})|>2^{k-1}\right)+\delta\\
&\quad\le2\delta.
\end{align*}
Since $2^M\lesssim\log^2(1/\delta)$ for $\delta<1/2$ (for $\delta\ge1/2$ the claim is trivial), with probability at least $1-2\delta$, it holds $R^w(\hat h_w)-R^w(\bar h_w)\lesssim r(\|w\|_2)^2\log^2(1/\delta)$ for all $w\in\Wcal$.
The lower bound follows from the norm comparison above, which holds for every $h\in\Hcal$, with $h=\hat h_w$.
\end{proof}

\paragraph{Simplifying Assumption}
To avoid cluttering the explicit rates below and in Appendices~\ref{app:linear-nonparametric} and~\ref{app:nonparametric}, we assume from now on:
\begin{itemize}
    \item[\textbf{(S)}] $C_1,B_\Wcal,C_\Pcal,K_\rho\lesssim1$, $\delta\ge n^{-1}$, and the data are exponentially $\beta$-mixing, $\beta(m)\lesssim e^{-cm}$ for some $c>0$.
\end{itemize}
Under (S), $(n/m)\beta(m)\le\delta$ holds for all $m\gtrsim\log(n/\delta)$, so that $m_\beta\lesssim\log n$.

\begin{proposition}\label{prop:cr-rate-simplify}
Assume the conditions of Theorem~\ref{thm:cr-fast-full}, (S), and \(C_\Wcal\ge n^{-1/2}\). Suppose that for
some \(D\ge1\) and \(\xi\in[0,2)\) with \(D\log n\le n\), and for all \(w\in\Wcal\) and \(\epsilon\in(0,1]\),
\begin{equation}\label{cond:cr-rate-complexities}
 \log N_\infty(\epsilon,\Hcal)
 +\log N_1(\epsilon,\Wcal)
 \lesssim D\|w\|_{2}^{-\xi}\log(n/\epsilon).
\end{equation}
Then there is a sufficiently large constant \(A\ge1\) for which
\[
 r(u)=u^{1-\xi/2}\sqrt{AD}\,\log n
\]
satisfies the rate condition in Theorem~\ref{thm:cr-fast-full}. In particular, with probability at least \(1-2\delta\), for all \(w\in\Wcal\),
\[
 R^w(\hat h_w)-R^w(\bar h_w)\lesssim D\|w\|_2^{2-\xi}\log^2n\,\log^2(1/\delta).
\]
\end{proposition}

\begin{proof}
Take \(K=A^2n^2\) in Theorem~\ref{thm:cr-fast-full}. Since
\(C_\Wcal\ge n^{-1/2}\), both covering radii in that theorem have inverses
bounded by a polynomial in \(n\). Hence
\eqref{cond:cr-rate-complexities}, with \(u=\|w\|_{2}\), gives
\(K_w\le A D u^{-\xi}\log n\)
after increasing \(A\). Using \(\min\{2,C_\Pcal C_\infty^{-1}r(u)\}\le2\) and (S), we have \(C_\Pcal^2K_\rho+m_\beta B_\Wcal\min\{2,C_\Pcal C_\infty^{-1}r(u)\}\lesssim\log n\), so the theorem's rate
condition is implied by
\[
 r(u)^2\ge A D u^{2-\xi}\log^2n
\]
after increasing \(A\), which is an equality for the stated choice. This choice is increasing because \(\xi<2\). Differentiating it,
using \(u\ge n^{-1/2}\) and \(D\log n\le n\), shows that its
Lipschitz constant is at most \(A^2n^2=K\). It is therefore admissible in
Theorem~\ref{thm:cr-fast-full}, which gives the final bound.
\end{proof}

\begin{proposition}\label{prop:squareloss-bayes}
For a distribution $P$, write $R_P(h)=\E_P[(Y-h(Z))^2]$.
    Let \(L(X,h)=(Y-h(Z))^2\) and $h^* = \E[Y|Z]$. Then
    \begin{align*}
        R_P(h)-R_P(h^*) = \|h-h^*\|_{L_2(P)}^2.
    \end{align*}
\end{proposition}

\begin{proof}
    For the square loss, the risk can be written as
\begin{align*}
    R_P(h) &= \E[(Y-h(Z))^2] \\
    &= \E[(Y-h^*(Z)+h^*(Z)-h(Z))^2] \\
    &= \E[(Y-h^*(Z))^2] + \E[(h(Z)-h^*(Z))^2] + 2\E[(Y-h^*(Z))(h^*(Z)-h(Z))] \\
    &= R_P(h^*) + \|h-h^*\|_{L_2(P)}^2 + 2\E[(h^*(Z)-h(Z)) \E[(Y-h^*(Z))\mid Z]] \\
    &= R_P(h^*) + \|h-h^*\|_{L_2(P)}^2,
\end{align*}
where the last equality follows from $\E[(Y-h^*(Z))\mid Z] = 0$ by definition of the Bayes rule.
\end{proof}

\subsection{Discrepancy}\label{app:discrepancy-proof}

\begin{proposition}\label{prop:discrepancy}
    For the oracle drift error, it holds
\begin{align*}
    R_t(\bar h_w)-R_t(\bar h_t) \leq 2 \mathrm{dis}(P_w,P_t).
\end{align*}
\end{proposition}

\begin{proof}
Since $R^w(\bar h_w)\leq R^w(\bar h_t)$,
\begin{align*}
R_t(\bar h_w)-R_t(\bar h_t)
&= R_t(\bar h_w)-R^w(\bar h_w) + R^w(\bar h_w)-R^w(\bar h_t) + R^w(\bar h_t)-R_t(\bar h_t) \\
&\le R_t(\bar h_w)-R^w(\bar h_w) + R^w(\bar h_t)-R_t(\bar h_t) \\
&\leq  |R^w(\bar h_t)-R_t(\bar h_t)| + |R^w(\bar h_w)-R_t(\bar h_w)| \\
&\le 2 \sup_{h\in\Hcal}
    |R^w(h)-R_t(h)|=2\,\mathrm{dis}(P_w,P_t).
\end{align*}
\end{proof}

\section{Linear Regression}\label{app:linear-nonparametric}

Let
\[
    \beta_w\in\argmin_{\beta\in\Theta}R^w(\beta),
    \qquad
    \beta_t\in\argmin_{\beta\in\Theta}R_t(\beta),
\]
so that \(\bar h_w=h_{\beta_w}\) and
\(\bar h_t=h_{\beta_t}\). As in Section~\ref{sec:linear_models}, the model is correctly specified, so $\bar h_t=h_t^*$ and $R_t(h)-R_t(\bar h_t)=\|h-\bar h_t\|_{L_2(P_t)}^2$ by Proposition~\ref{prop:squareloss-bayes}.
Define $a_i=\E_{P_i}[Y^2]$, $G_i=\E_{P_i}[ZZ^\top]$, and $b_i=\E_{P_i}[ZY]$. We will always assume that for some \(0<\gamma\le \Gamma<\infty\),
\[
    \gamma I_p
    \preceq
    G_t
    \preceq
    \Gamma I_p .
\]
Further, define the norm $\|v\|_{G_t}=\sqrt{v^\top G_t v}$.
\begin{lemma}\label{lem:squarelosslinear}
    For the square loss, it holds
    \begin{align*}
        \|h_{\beta_w}-h_{\beta_t}\|_{L_2(P_t)}^2 = (\beta_w-\beta_t)^\top G_t (\beta_w-\beta_t) = \|\beta_w - \beta_t\|^2_{G_t}.
    \end{align*}
\end{lemma}

\begin{proof}
    The proof follows from a straightforward computation,
    \begin{align*}
         \|h_{\beta_w}-h_{\beta_t}\|_{L_2(P_t)}^2 &= \E_{P_t} [(h_{\beta_w}-h_{\beta_t})^2] \\
         &= \E_{P_t} [((\beta_w-\beta_t)^\top Z)^2] \\
         &= \E_{P_t} [(\beta_w-\beta_t)^\top ZZ^\top (\beta_w-\beta_t)] \\
         &= (\beta_w-\beta_t)^\top \E_{P_t} [ZZ^\top] (\beta_w-\beta_t) \\
         &= (\beta_w-\beta_t)^\top G_t (\beta_w-\beta_t) \\
         &= \|\beta_w - \beta_t\|^2_{G_t}.
    \end{align*}
\end{proof}

\subsection{Oracle Drift}

\begin{proof}[Proof of Proposition~\ref{prop:localnormalcontrol}]
We have
\[
    \|\nabla R^w(\beta_w)-\nabla R_t(\beta_w)\|_2
    =
    2\|(G_w-G_t)\beta_w-(b_w-b_t)\|_2
    =
    2\|v_w\|_2.
\]

Moreover,
\begin{align*}
\left\langle
\nabla R_t(\beta)-\nabla R_t(\beta_t),
\beta-\beta_t
\right\rangle
&=
2(\beta-\beta_t)^\top G_t(\beta-\beta_t) \ge
2\gamma\|\beta-\beta_t\|_2^2 .
\end{align*}
Thus Theorem~\ref{thm:parametric_oracle_stability_local} applies with curvature
constant \(2\gamma\) and gradient discrepancy \(\Delta_w=2\|v_w\|_2\), and its proof gives $\|\beta_w-\beta_t\|_2\le\|v_w\|_2/\gamma$.
By Lemma \ref{lem:squarelosslinear},
\begin{align*}
\|\bar h_w-\bar h_t\|_{L_2(P_t)}^2
&=
\|h_{\beta_w}-h_{\beta_t}\|_{L_2(P_t)}^2 =
(\beta_w-\beta_t)^\top
G_t
(\beta_w-\beta_t) \le
\Gamma
\|\beta_w-\beta_t\|_2^2 \le
\frac{\Gamma}{\gamma^2}
\|v_w\|_2^2 .
\end{align*}
\end{proof}

\begin{proof}[Proof of Proposition~\ref{thm:driftoptimality}]
Since $R_i(\beta)=a_i-2\beta^\top b_i+\beta^\top G_i \beta$ and $R^w(\beta)=a_w-2\beta^\top b_w+\beta^\top G_w \beta$, the gradients are
\begin{align*}
   \nabla R_t(\beta) = 2(G_t\beta-b_t), \qquad \nabla R^w(\beta) = 2(G_w \beta-b_w).
\end{align*}
Therefore, the first-order conditions $\nabla R^w(\beta_w)=0$ and $\nabla R_t(\beta_t)=0$ give $G_w \beta_w = b_w$ and $G_t\beta_t = b_t$. Now, write
\[
\begin{aligned}
v_w &= (G_w-G_t)\beta_w-(b_w-b_t) \\
&=(G_w-G_t)\beta_w-(G_w\beta_w-G_t\beta_t) \\
&=G_w\beta_w-G_t\beta_w-G_w\beta_w+G_t\beta_t \\
&=G_t\beta_t-G_t\beta_w \\
&=G_t(\beta_t-\beta_w).
\end{aligned}
\]
Thus \(\beta_w-\beta_t=-G_t^{-1}v_w\), and by Lemma \ref{lem:squarelosslinear} it holds
\[
\|\bar h_w-\bar h_t\|_{L_2(P_t)}^2
=
(\beta_w-\beta_t)^\top G_t(\beta_w-\beta_t)
=
v_w^\top G_t^{-1}v_w .
\]
The eigenvalue bounds on \(G_t\) give the result.
\end{proof}

\subsection{Global and Localized Discrepancy}

\begin{proposition}
\label{prop:discrepancy_upper_local_moments}
Assume that \(\Theta\subseteq B\mathbb B_2^p\) and \(\|\beta_w\|_2\le B\). Let $ v_w=(G_w-G_t)\beta_w-(b_w-b_t)$.
Then
\[
    \mathrm{dis}(P_w,P_t)
    \le
    |a_w-a_t|
    +
    2B\|v_w\|_2
    +
    3B^2\|G_w-G_t\|_{\mathrm{op}} .
\]
\end{proposition}

\begin{proof}
For square loss,
\[
    R^w(\beta)-R_t(\beta)
    =
    (a_w-a_t)-2\beta^\top(b_w-b_t)
    +\beta^\top(G_w-G_t)\beta .
\]
Since
\[
    b_w-b_t=(G_w-G_t)\beta_w-v_w,
\]
we obtain
\[
\begin{aligned}
    R^w(\beta)-R_t(\beta)
    &=
    (a_w-a_t)
    -2\beta^\top (G_w-G_t)\beta_w
    +2\beta^\top v_w
    +\beta^\top(G_w-G_t)\beta .
\end{aligned}
\]
Thus, for every \(\beta\in\Theta\),
\[
\begin{aligned}
    |R^w(\beta)-R_t(\beta)|
    &\le
    |a_w-a_t|
    +2\|\beta\|_2\|v_w\|_2
    +2\|\beta\|_2\|G_w-G_t\|_{\mathrm{op}}\|\beta_w\|_2
    +\|\beta\|_2^2\|G_w-G_t\|_{\mathrm{op}} \\
    &\le
    |a_w-a_t|
    +2B\|v_w\|_2
    +3B^2\|G_w-G_t\|_{\mathrm{op}} .
\end{aligned}
\]
Taking the supremum over \(\beta\in\Theta\) gives the claim.
\end{proof}

\paragraph{Lower Bounds on the Discrepancy}
We now show that the discrepancy is, in general, at
least first order in the same quantity.

Let $a_w=\sum_{i=1}^{t-1}w_ia_i$.
Then
\[
    \mathrm{dis}(P_w,P_t)
    =
    \sup_{\beta\in\Theta}
    \left|
        a_w-a_t - 2\beta^\top (b_w-b_t) + \beta^\top (G_w-G_t)\beta
    \right|.
\]

\begin{lemma}
\label{lem:discrepancy_lower_local}
Let $v_w  = (G_w-G_t)\beta_w-(b_w-b_t)$. Assume \(v_w\neq 0\), and suppose that for some \(\rho_w>0\),
\[
    \beta_w \pm \rho_w \frac{v_w}{\|v_w\|_2}
    \in \Theta .
\]
Then
\[
    \mathrm{dis}(P_w,P_t)
    \ge
    2\rho_w
    \|v_w\|_2 .
\]
\end{lemma}

\begin{proof}
Let
\[
    f(\beta)=R^w(\beta)-R_t(\beta)
    =
    a_w-a_t - 2\beta^\top (b_w-b_t) + \beta^\top (G_w-G_t) \beta ,
\]
and set \(u=v_w/\|v_w\|_2\). Then
\[
\begin{aligned}
    f(\beta_w+\rho_w u)-f(\beta_w-\rho_w u)
    &=
    4\rho_w u^\top((G_w-G_t) \beta_w-(b_w-b_t))  \\
    &=
    4\rho_w \|v_w\|_2 .
\end{aligned}
\]
Therefore,
\[
\begin{aligned}
    \mathrm{dis}(P_w,P_t)
    &\ge
    \max\{
        |f(\beta_w+\rho_w u)|,
        |f(\beta_w-\rho_w u)|
    \}  \\
    &\ge
    \frac12
    |f(\beta_w+\rho_w u)-f(\beta_w-\rho_w u)|  \\
    &=
    2\rho_w \|v_w\|_2 .
\end{aligned}
\]
\end{proof}

\begin{lemma}
\label{lem:discrepancy_lower_moments}
Assume \(0\in\Theta\). Then
\[
    \mathrm{dis}(P_w,P_t)
    \ge
    |a_w-a_t|.
\]
If, in addition, \(r\mathbb B_2^p = \{\beta \in \mathbb{R}^p : \|\beta\|_2 \le r\} \subseteq\Theta\), then
\[
    \mathrm{dis}(P_w,P_t)
    \ge
    2r\|b_w-b_t\|_2, \qquad     \mathrm{dis}(P_w,P_t)
    \ge
    \frac{r^2}{2}\|G_w - G_t\|_{\mathrm{op}} .
\]
\end{lemma}

\begin{proof}
The first claim follows by evaluating the discrepancy at \(\beta=0\). For the
second claim, if \(b_w\ne b_t\), take \(u=(b_w-b_t)/\|b_w-b_t\|_2\); otherwise the claim is trivial. Since \(ru,-ru\in\Theta\),
\[
    f(ru)-f(-ru)=-4r u^\top (b_w-b_t)=-4r\|b_w-b_t\|_2.
\]
Hence
\[
    \mathrm{dis}(P_w,P_t)
    \ge
    \frac12|f(ru)-f(-ru)|
    =
    2r\|b_w-b_t\|_2.
\]
For the third claim, let \(u\) be a unit eigenvector corresponding to an
eigenvalue of \(G_w - G_t\) with absolute value
\(\|G_w - G_t\|_{\mathrm{op}}\). Since
\[
    \frac{f(ru)+f(-ru)}{2}
    =
    a_w-a_t+r^2u^\top(G_w - G_t) u,
\]
we have
\[
    \max\{|f(ru)|,|f(-ru)|\}
    \ge
    |a_w-a_t+r^2u^\top (G_w - G_t) u|.
\]
Combining this with \(|f(0)|=|a_w-a_t|\) gives
\[
    \mathrm{dis}(P_w,P_t)
    \ge
    \frac{r^2}{2}\|G_w - G_t\|_{\mathrm{op}} .
\]
\end{proof}

\begin{corollary}
    Under the assumptions of Lemmas \ref{lem:discrepancy_lower_local} and \ref{lem:discrepancy_lower_moments}, it holds
    \begin{align*}
        \mathrm{dis}(P_w,P_t) \ge \max \left\{2\rho_w\|v_w\|_2, |a_w-a_t|, 2r\|b_w-b_t\|_2, \frac{r^2}{2}\|G_w - G_t\|_{\mathrm{op}}\right\}.
    \end{align*}
\end{corollary}

\begin{proof}
    The result follows by taking the maximum of the individual terms.
\end{proof}

\begin{proposition}[Full version of Proposition~\ref{prop:discrepancy-asymp}]\label{prop:discrepancy-full}
Let $\Theta=B\mathbb B_2^p$ and assume $\|\beta_w\|_2\le B-\rho$ for some $\rho\in(0,B)$. Then
\begin{align*}
 \max\Big\{|a_w-a_t|,\,2\rho\|v_w\|_2,\,\tfrac{B^2}{2}\|G_w-G_t\|_{\mathrm{op}}\Big\}
 &\le\mathrm{dis}(P_w,P_t)\\
 &\le|a_w-a_t|+2B\|v_w\|_2+3B^2\|G_w-G_t\|_{\mathrm{op}}.
\end{align*}
\end{proposition}
\begin{proof}
The upper bound is Proposition~\ref{prop:discrepancy_upper_local_moments}. For the lower bound, $0\in\Theta$ and $B\mathbb B_2^p\subseteq\Theta$, so Lemma~\ref{lem:discrepancy_lower_moments} applies with $r=B$ and gives $\mathrm{dis}(P_w,P_t)\ge|a_w-a_t|$ and $\mathrm{dis}(P_w,P_t)\ge\frac{B^2}{2}\|G_w-G_t\|_{\mathrm{op}}$. Since $\|\beta_w\|_2\le B-\rho$, we have $\beta_w\pm\rho u\in\Theta$ for every unit vector $u$, so Lemma~\ref{lem:discrepancy_lower_local} with $\rho_w=\rho$ gives $\mathrm{dis}(P_w,P_t)\ge2\rho\|v_w\|_2$ if $v_w\ne0$; this bound is trivial if $v_w=0$.

Finally, write $S_w=|a_w-a_t|+\|v_w\|_2+\|G_w-G_t\|_{\mathrm{op}}$. Since the maximum of three nonnegative numbers is at least their average,
\[
 \tfrac13\min\{1,2\rho,\tfrac{B^2}{2}\}\,S_w
 \le\mathrm{dis}(P_w,P_t)
 \le\max\{1,2B,3B^2\}\,S_w,
\]
which proves Proposition~\ref{prop:discrepancy-asymp}.
\end{proof}

\begin{proposition}[Full version of Proposition~\ref{prop:linear-localized-discrepancy}]\label{prop:localized-full}
Let $\beta_w$ be an interior oracle, $G_w\succ0$, $\ell_w=R^w(\beta_w)-R_t(\beta_w)$, and $\tilde\Lambda=G_w^{-1/2}(G_w-G_t)G_w^{-1/2}$.
\begin{enumerate}[label=(\roman*)]
\item For every $\varepsilon\ge0$,
\[
 \mathrm{dis}_\varepsilon^{\mathrm{loc}}(P_w,P_t)
 \le|\ell_w|+2\sqrt{\varepsilon}\|G_w^{-1/2}v_w\|_2+\varepsilon\|\tilde\Lambda\|_{\mathrm{op}}.
\]
\item If in addition $\{\beta_w+\tau:\tau^\top G_w\tau\le\varepsilon\}\subseteq\Theta$, then
\[
 \mathrm{dis}_\varepsilon^{\mathrm{loc}}(P_w,P_t)
 \ge\max\Big\{|\ell_w|,\,2\sqrt{\varepsilon}\|G_w^{-1/2}v_w\|_2,\,\tfrac{\varepsilon}{2}\|\tilde\Lambda\|_{\mathrm{op}}\Big\}.
\]
\end{enumerate}
\end{proposition}
\begin{proof}[Proof of Proposition~\ref{prop:localized-full}(i)]
Write $f(\beta)=R^w(\beta)-R_t(\beta)=(a_w-a_t)-2\beta^\top(b_w-b_t)+\beta^\top(G_w-G_t)\beta$. For $\beta=\beta_w+\tau$,
\[
 f(\beta_w+\tau)
 =f(\beta_w)+2\tau^\top\big((G_w-G_t)\beta_w-(b_w-b_t)\big)+\tau^\top(G_w-G_t)\tau,
\]
that is, with $\ell_w=f(\beta_w)$ and $v_w=(G_w-G_t)\beta_w-(b_w-b_t)$,
\begin{equation}\label{eq:linear-localized-difference}
R^w(\beta)-R_t(\beta)
 =\ell_w+2\tau^\top v_w+\tau^\top(G_w-G_t)\tau.
\end{equation} Since $\beta_w$ is interior, $G_w\beta_w=b_w$ and
\[
 R^w(\beta)-R^w(\beta_w)=2\tau^\top(G_w\beta_w-b_w)+\tau^\top G_w\tau=\tau^\top G_w\tau .
\]
Hence every $h_\beta\in\Hcal_{w,\varepsilon}$ satisfies $\tau^\top G_w\tau\le\varepsilon$. Set $u=G_w^{1/2}\tau$, so that $\|u\|_2^2\le\varepsilon$. The linear and quadratic terms satisfy
\begin{align*}
|\tau^\top v_w|
&\le\sqrt{\varepsilon}\|G_w^{-1/2}v_w\|_2,\\
|\tau^\top(G_w-G_t)\tau|
&\le\varepsilon\,
\|G_w^{-1/2}(G_w-G_t)G_w^{-1/2}\|_{\mathrm{op}},
\end{align*}
which proves (i).
\end{proof}

\begin{proof}[Proof of Proposition~\ref{prop:localized-full}(ii)]
Let $F(\tau)$ denote the right-hand side of \eqref{eq:linear-localized-difference} and $\tilde\Lambda=G_w^{-1/2}(G_w-G_t)G_w^{-1/2}$. Since $\{\beta_w+\tau:\tau^\top G_w\tau\le\varepsilon\}\subseteq\Theta$ by assumption and $R^w(\beta)-R^w(\beta_w)=\tau^\top G_w\tau$ by the proof of (i), $\Hcal_{w,\varepsilon}$ is the complete ellipsoid, so $\mathrm{dis}_\varepsilon^{\mathrm{loc}}(P_w,P_t)=\sup\{|F(\tau)|:\tau^\top G_w\tau\le\varepsilon\}$.

First, $\tau=0$ is feasible and $|F(0)|=|\ell_w|$.

Second, if $v_w\ne0$, let
\[
u_*=\sqrt{\varepsilon}\,
\frac{G_w^{-1/2}v_w}{\|G_w^{-1/2}v_w\|_2},
\qquad
\tau_\pm=\pm G_w^{-1/2}u_* .
\]
Both $\tau_\pm$ are feasible, and the constant and quadratic terms of $F$ agree at $\tau_+$ and $\tau_-$. Hence
\[
\max\{|F(\tau_+)|,|F(\tau_-)|\}
\ge\frac12|F(\tau_+)-F(\tau_-)|
=2\,\tau_+^\top v_w
=2\sqrt{\varepsilon}\|G_w^{-1/2}v_w\|_2.
\]
If $v_w=0$, this bound is trivial.

Third, let $u$ be a unit eigenvector of the symmetric matrix $\tilde\Lambda$ for an eigenvalue $\lambda$ with $|\lambda|=\|\tilde\Lambda\|_{\mathrm{op}}$, and let $\tau_\pm=\pm\sqrt{\varepsilon}\,G_w^{-1/2}u$. Both are feasible, the linear terms of $F$ cancel in the sum, and
\[
F(\tau_+)+F(\tau_-)
=2\ell_w+2\varepsilon\lambda .
\]
Therefore
\[
\mathrm{dis}_\varepsilon^{\mathrm{loc}}(P_w,P_t)
\ge\frac12|F(\tau_+)+F(\tau_-)|
\ge\varepsilon\|\tilde\Lambda\|_{\mathrm{op}}-|\ell_w| .
\]
Combining this with the first bound gives
\[
 \mathrm{dis}_\varepsilon^{\mathrm{loc}}(P_w,P_t)\ge\max\{|\ell_w|,\varepsilon\|\tilde\Lambda\|_{\mathrm{op}}-|\ell_w|\}\ge\tfrac{\varepsilon}{2}\|\tilde\Lambda\|_{\mathrm{op}},
\]
which proves (ii).

Finally, by (i) and (ii),
\begin{align*}
 \tfrac16\big(|\ell_w|+2\sqrt{\varepsilon}\|G_w^{-1/2}v_w\|_2+\varepsilon\|\tilde\Lambda\|_{\mathrm{op}}\big)
 &\le\tfrac13\big(|\ell_w|+2\sqrt{\varepsilon}\|G_w^{-1/2}v_w\|_2+\tfrac{\varepsilon}{2}\|\tilde\Lambda\|_{\mathrm{op}}\big)
 \le\mathrm{dis}_\varepsilon^{\mathrm{loc}}(P_w,P_t)\\
 &\le|\ell_w|+2\sqrt{\varepsilon}\|G_w^{-1/2}v_w\|_2+\varepsilon\|\tilde\Lambda\|_{\mathrm{op}},
\end{align*}
which proves Proposition~\ref{prop:linear-localized-discrepancy}.
\end{proof}

\subsection{Drift Scenarios}

\subsubsection{Feature Drift}\label{app:feature-moment-proof}
We consider the setting of Section~\ref{sec:feature-moment-drift}: $Y_i=\beta^{*\top}Z_i+\eta_i$ with $\E[\eta_i\mid Z_i]=0$, $\E[\eta_i^2]=\sigma^2$, and $\|\beta^*\|_2<B$, while $G_i$ may change over time. Write $\Lambda=G_w-G_t$.

\begin{proposition}\label{prop:feature-drift-full}
Let $G_w\succ0$ and $G_t\succ0$. Then $\beta_w=\beta_t=\beta^*$, so the oracle drift is zero, and
\[
 (B-\|\beta^*\|_2)^2\|\Lambda\|_{\mathrm{op}}
 \le\mathrm{dis}(P_w,P_t)
 \le(B+\|\beta^*\|_2)^2\|\Lambda\|_{\mathrm{op}}.
\]
If $\varepsilon\le\lambda_{\min}(G_w)(B-\|\beta^*\|_2)^2$, then
\[
 \mathrm{dis}_\varepsilon^{\mathrm{loc}}(P_w,P_t)
 =\varepsilon\|G_w^{-1/2}\Lambda G_w^{-1/2}\|_{\mathrm{op}}.
\]
\end{proposition}

\begin{proof}
\emph{Oracles.} Since $b_i=G_i\beta^*$ and $a_i=\sigma^2+\beta^{*\top}G_i\beta^*$,
\[
 R_i(\beta)=a_i-2\beta^\top b_i+\beta^\top G_i\beta
 =\sigma^2+(\beta-\beta^*)^\top G_i(\beta-\beta^*),
\]
and the same identity holds for $R^w$ with $G_w$. As $G_w,G_t\succ0$ and $\beta^*\in\Theta$, both risks are uniquely minimized at $\beta^*$, so $\beta_w=\beta_t=\beta^*$.

\emph{Discrepancy.} By the previous display,
\[
 R^w(\beta)-R_t(\beta)=(\beta-\beta^*)^\top \Lambda(\beta-\beta^*).
\]
For $\|\beta\|_2\le B$, we have $\|\beta-\beta^*\|_2\le B+\|\beta^*\|_2$, which gives the upper bound. For the lower bound, let $u$ be a unit eigenvector of $\Lambda$ for an eigenvalue $\lambda$ with $|\lambda|=\|\Lambda\|_{\mathrm{op}}$, and set $\beta=\beta^*+(B-\|\beta^*\|_2)u$. Then $\|\beta\|_2\le B$ and
\[
 |R^w(\beta)-R_t(\beta)|=(B-\|\beta^*\|_2)^2\|\Lambda\|_{\mathrm{op}}.
\]

\emph{Localized discrepancy.} Write $\tau=\beta-\beta^*$. Then $R^w(\beta)-R^w(\beta^*)=\tau^\top G_w\tau$, so $\Hcal_{w,\varepsilon}$ is the feasible part of the ellipsoid $\{\tau^\top G_w\tau\le\varepsilon\}$. On this ellipsoid, $\|\tau\|_2^2\le\varepsilon/\lambda_{\min}(G_w)\le(B-\|\beta^*\|_2)^2$, so it lies entirely in $\Theta$. With $u=G_w^{1/2}\tau$,
\[
 \mathrm{dis}_\varepsilon^{\mathrm{loc}}(P_w,P_t)
 =\sup_{\|u\|_2^2\le\varepsilon}
 |u^\top G_w^{-1/2}\Lambda G_w^{-1/2}u|
 =\varepsilon\|G_w^{-1/2}\Lambda G_w^{-1/2}\|_{\mathrm{op}},
\]
where the last step uses the absolute Rayleigh quotient.
\end{proof}

Proposition~\ref{prop:feature-moment-drift} is the last statement of Proposition~\ref{prop:feature-drift-full}.

\subsubsection{Smooth Drift}
Let $ \Hcal=\{h_\beta(z)=\beta^\top z:\|\beta\|_2\le B\}$ and assume the square-loss linear model $Y_i=\beta_i^{*\top}Z_i+\eta_i$ with
$\E[\eta_i\mid Z_i]=0$ and $\|\beta_i^*\|_2\le B$, and define $ G_i=\E_{P_i}[ZZ^\top]$, $b_i=\E_{P_i}[ZY]$.

Consider the smooth-drift setting where $\psi_i(t)=|t-i|^\nu$.

\begin{proof}[Proof of Proposition~\ref{prop:smooth-local}]
By the triangle inequality,
\begin{align*}
\|(G_w-G_t)\beta_w-(b_w-b_t)\|_2
&=
\left\|
\sum_{i=1}^{t-1}w_i
((G_i-G_t)\beta_w-(b_i-b_t))
\right\|_2 \\
&\le
\sum_{i=1}^{t-1} |w_i|
\|(G_i-G_t)\beta_w-(b_i-b_t)\|_2 \\
&\le
L_v
\sum_{i=1}^{t-1}|w_i| |t-i|^\nu.
\end{align*}
Since the left-hand side is $\|v_w\|_2$, Proposition~\ref{prop:localnormalcontrol} gives
\[
    (R_t(\bar h_w)-R_t(\bar h_t))^{1/2}
    \lesssim
    \|\bar h_w-\bar h_t\|_{L_2(P_t)}
    \le
    \frac{\sqrt{\Gamma}}{\gamma}
    L_v\sum_{i=1}^{t-1}|w_i| |t-i|^\nu.
\]
\end{proof}

\begin{remark}\label{rem:smooth-local-sufficient}
If $\|G_i-G_t\|_{\mathrm{op}}\le L_G|t-i|^\nu$ and $\|\beta_i^*-\beta_t^*\|_2\le L_\beta|t-i|^\nu$ for all $i<t$, the assumption of Proposition~\ref{prop:smooth-local} holds with $L_v=2BL_G+\Gamma L_\beta$. Indeed, $\E[\eta_i\mid Z_i]=0$ gives $b_i=G_i\beta_i^*$, so $b_i-b_t=(G_i-G_t)\beta_i^*+G_t(\beta_i^*-\beta_t^*)$. Since $\|\beta_w\|_2,\|\beta_i^*\|_2\le B$,
\[
\|(G_i-G_t)\beta_w-(b_i-b_t)\|_2
\le 2B\|G_i-G_t\|_{\mathrm{op}}+\Gamma\|\beta_i^*-\beta_t^*\|_2
\le(2BL_G+\Gamma L_\beta)|t-i|^\nu.
\]
\end{remark}

\begin{proof}[Proof of Proposition~\ref{prop:smooth-discrepancy}]
For square loss, $ R_i(\beta)=a_i-2\beta^\top b_i+\beta^\top G_i\beta$.
Thus, with $ a_w=\sum_{i=1}^{t-1}w_ia_i,$
we have
\[
    R^w(\beta)-R_t(\beta)
    =
    (a_w-a_t)
    -
    2\beta^\top(b_w-b_t)
    +
    \beta^\top(G_w-G_t)\beta .
\]
By the smooth-drift assumptions,
\[
    |a_w-a_t|
    \le
    L_aM_w(\psi;t),
    \qquad
    \|G_w-G_t\|_{\mathrm{op}}
    \le
    L_G M_w(\psi;t),
    \qquad
    \|b_w-b_t\|_2
    \le
    L_b M_w(\psi;t).
\]
Therefore, for every \(\|\beta\|_2\le B\),
\begin{align*}
|R^w(\beta)-R_t(\beta)|
&\le
|a_w-a_t|
+
2|\beta^\top(b_w-b_t)|
+
|\beta^\top(G_w-G_t)\beta| \\
&\le
\left(
    L_a+2BL_b+B^2L_G
\right)
M_w(\psi;t).
\end{align*}
Taking the supremum over $\beta\in\Theta$ gives the same bound for $\mathrm{dis}(P_w,P_t)$.
Combining this with Proposition~\ref{prop:discrepancy} gives the result.
\end{proof}

\subsection{Estimation Error}\label{app:linear-estimation}

For \(\Hcal=\{h_\beta(z)=\beta^\top z:\beta\in\Theta\}\), the following statement makes the
dimension dependence in the linear-model estimation rate explicit.
Since $|Y|\le B$, $\|\beta\|_2\le B$, and $\|Z\|_2\le1$, the square loss is bounded by $4B^2$ and $4B$-Lipschitz in the prediction. The unit normalizations in \eqref{cond:cr-simplify} may therefore be replaced by these fixed $B$-dependent bounds.

The covering numbers below follow from the following standard bound.
\begin{lemma}\label{lem:norm-ball-covering}
Let $\|\cdot\|$ be a norm on $\mathbb R^d$ and $S=\{x\in\mathbb R^d:\|x\|\le r\}$ for some $r>0$. Then, for every $\epsilon>0$,
\[
 N(\epsilon,S,\|\cdot\|)\le\Bigl(1+\frac{2r}{\epsilon}\Bigr)^d.
\]
\end{lemma}
\begin{proof}
See, e.g., \citet[Section~4.2]{vershynin2018high}.
\end{proof}

\begin{proposition}\label{prop:linear-estimation}
Suppose the conditions of Theorem~\ref{thm:cr-fast-full} and (S) hold,
\(\sup_z\|z\|_2\le1\), \(\Theta\subseteq B\mathbb B_2^p\), \(\log N_1(\epsilon,\Wcal)\lesssim\log(n/\epsilon)\), \(C_\Wcal\ge n^{-1/2}\), and \(p\log n\le n\). Then, with probability at least \(1-2\delta\), for every
\(w\in\Wcal\),
\[
 R^w(\hat h_w)-R^w(\bar h_w)
 \lesssim p\|w\|_2^2\log^2n\,\log^2(1/\delta).
\]
\end{proposition}

\begin{proof}
{ For \(h_\beta,h_{\beta'}\in\Hcal\), \(\|h_\beta-h_{\beta'}\|_\infty\le\|\beta-\beta'\|_2\), so every \(\epsilon\)-cover of \(B\mathbb B_2^p\) in \(\|\cdot\|_2\) induces an \(\epsilon\)-cover of \(\Hcal\) in \(\|\cdot\|_\infty\). Lemma~\ref{lem:norm-ball-covering} with the Euclidean norm and \(r=B\) therefore gives
\[
 \log N_\infty(\epsilon,\Hcal)
 \le p\log(1+2B/\epsilon).
\]
Thus \eqref{cond:cr-rate-complexities} holds with \(D\asymp p\) and
\(\xi=0\), and Proposition~\ref{prop:cr-rate-simplify} proves the claim.}
\end{proof}

\subsection{Excess-Risk Bounds}\label{app:linear-excess-risk}
Throughout this subsection, the sample consists of the $n=t-1$ past observations, the linear model is correctly specified, so that $\bar h_t=h_t^*$, and the conditions of Proposition~\ref{prop:linear-estimation} hold.

\begin{lemma}\label{lem:linear-excess-decomposition}
With probability at least $1-2\delta$, it holds for all $w\in\Wcal$,
\[
 R_t(\hat h_w)-R_t(h_t^*)
 \lesssim
 p\|w\|_2^2\log^2n\,\log^2(1/\delta)+R_t(\bar h_w)-R_t(\bar h_t).
\]
\end{lemma}

\begin{proof}
Since $\bar h_t=h_t^*$, Proposition~\ref{prop:squareloss-bayes} gives
\begin{align*}
 R_t(\hat h_w)-R_t(h_t^*)
 =\|\hat h_w-\bar h_t\|_{L_2(P_t)}^2 \le
 2\|\hat h_w-\bar h_w\|_{L_2(P_t)}^2
 +2\|\bar h_w-\bar h_t\|_{L_2(P_t)}^2,
\end{align*}
and $\|\bar h_w-\bar h_t\|_{L_2(P_t)}^2=R_t(\bar h_w)-R_t(\bar h_t)$. Since $P_t\in\Pcal$, Theorem~\ref{thm:cr-fast-full} and Proposition~\ref{prop:linear-estimation} give
\[
 \|\hat h_w-\bar h_w\|_{L_2(P_t)}^2
 \le\|\hat h_w-\bar h_w\|_{L_2(\Pcal)}^2
 \lesssim C_\Pcal^2\{R^w(\hat h_w)-R^w(\bar h_w)\}
 \lesssim p\|w\|_2^2\log^2n\,\log^2(1/\delta).
\]
\end{proof}

\begin{corollary}\label{cor:excess-noise}
In the setting of Section~\ref{sec:noise-drift}, let $w_i=n^{-1}$ for all $i<t$. Then, with probability at least $1-2\delta$,
\[
 R_t(\hat h_w)-R_t(h_t^*)
 \lesssim
 \frac{p\log^2n\,\log^2(1/\delta)}{n}.
\]
Bounding the oracle drift by \eqref{eq:decomposition_discrepancy} instead gives
\[
 R_t(\hat h_w)-R_t(h_t^*)
 \lesssim
 \frac{p\log^2n\,\log^2(1/\delta)}{n}
 +\Big|\frac{1}{n}\sum_{i<t}\sigma_i^2-\sigma_t^2\Big|.
\]
\end{corollary}

\begin{proof}
The conditional mean does not change, so $\bar h_w=\bar h_t=h_t^*$ and the oracle drift is zero. Since $\|w\|_2^2=n^{-1}$, Lemma~\ref{lem:linear-excess-decomposition} gives the first bound. For the second bound, $G_i=G_t$ and $b_i=b_t$ for all $i<t$, so
\[
 R^w(\beta)-R_t(\beta)=a_w-a_t=\frac{1}{n}\sum_{i<t}\sigma_i^2-\sigma_t^2
 \qquad\text{for all }\beta\in\Theta,
\]
and \eqref{eq:decomposition_discrepancy} bounds the oracle drift by twice the absolute value of this constant.
\end{proof}

\begin{corollary}\label{cor:excess-feature}
In the setting of Section~\ref{sec:feature-moment-drift}, assume $G_i\succ0$ for all $i<t$ and let $w_i=n^{-1}$ for all $i<t$. Then, with probability at least $1-2\delta$, $R_t(\hat h_w)-R_t(h_t^*)\lesssim p\log^2n\,\log^2(1/\delta)/n$. With the discrepancies, we get
\[
 R_t(\hat h_w)-R_t(h_t^*)
 \lesssim
 \frac{p\log^2n\,\log^2(1/\delta)}{n}
 +(B+\|\beta^*\|_2)^2\|G_w-G_t\|_{\mathrm{op}},
\]
and, if $\varepsilon\le\lambda_{\min}(G_w)(B-\|\beta^*\|_2)^2$, bounding it by \eqref{eq:decomposition_localized} gives
\[
 R_t(\hat h_w)-R_t(h_t^*)
 \lesssim
 \frac{p\log^2n\,\log^2(1/\delta)}{n}
 +\varepsilon\|G_w^{-1/2}(G_w-G_t)G_w^{-1/2}\|_{\mathrm{op}}.
\]
\end{corollary}

\begin{proof}
By Proposition~\ref{prop:feature-drift-full}, $\beta_w=\beta_t=\beta^*$, so the oracle drift is zero and Lemma~\ref{lem:linear-excess-decomposition} gives the first bound. The second bound follows from \eqref{eq:decomposition_discrepancy} and the upper bound on the discrepancy in Proposition~\ref{prop:feature-drift-full}. Since $\bar h_t=\bar h_w\in\Hcal_{w,\varepsilon}$, \eqref{eq:decomposition_localized} applies for every $\varepsilon\ge0$, and the third bound follows from the localized discrepancy in Proposition~\ref{prop:feature-drift-full}.
\end{proof}

\begin{corollary}\label{cor:excess-smooth}
With probability at least $1-2\delta$, the following holds for all $w\in\Wcal$.
\begin{enumerate}[label=(\roman*)]
\item If $\|G_i-G_t\|_{\mathrm{op}}\le L_G\psi_i(t)$ and $\|\beta_i^*-\beta_t^*\|_2\le L_\beta\psi_i(t)$ for all $i<t$,
\[
 R_t(\hat h_w)-R_t(h_t^*)
 \lesssim
 p\|w\|_2^2\log^2n\,\log^2(1/\delta)
 +\frac{\Gamma(2BL_G+\Gamma L_\beta)^2}{\gamma^2}
 M_w(\psi;t)^2.
\]
\item Under the conditions of Proposition~\ref{prop:smooth-discrepancy},
\[
 R_t(\hat h_w)-R_t(h_t^*)
 \lesssim
 p\|w\|_2^2\log^2n\,\log^2(1/\delta)
 +(L_a+B^2L_G+2BL_b)M_w(\psi;t).
\]
\item Under the conditions of (i) and (ii), assume further that $\gamma I_p\preceq G_w$, that $\beta_w$ is interior, and that $\varepsilon\in(0,1]$ satisfies $\bar h_t\in\Hcal_{w,\varepsilon}$. Then
\[
 R_t(\hat h_w)-R_t(h_t^*)
 \lesssim
 p\|w\|_2^2\log^2n\,\log^2(1/\delta)
 +\Big(L_a+2BL_b+B^2L_G+\sqrt{\frac{\varepsilon}{\gamma}}\Big(\frac{L_G}{\sqrt\gamma}+2BL_G+\Gamma L_\beta\Big)\Big)M_w(\psi;t).
\]
\end{enumerate}
\end{corollary}

\begin{proof}
Parts (i) and (ii) follow from Lemma~\ref{lem:linear-excess-decomposition} together with Proposition~\ref{prop:smooth-local} and Remark~\ref{rem:smooth-local-sufficient} for (i), and Proposition~\ref{prop:smooth-discrepancy} for (ii).

For (iii), $\bar h_t\in\Hcal_{w,\varepsilon}$, so \eqref{eq:decomposition_localized} and Proposition~\ref{prop:localized-full}(i) give
\[
 R_t(\bar h_w)-R_t(\bar h_t)
 \le
 2|\ell_w|+4\sqrt{\varepsilon}\|G_w^{-1/2}v_w\|_2+2\varepsilon\|G_w^{-1/2}(G_w-G_t)G_w^{-1/2}\|_{\mathrm{op}}.
\]
We bound the three terms separately. Since $\ell_w=\sum_{i<t}w_i\{R_i(\beta_w)-R_t(\beta_w)\}$ and $\|\beta_w\|_2\le B$,
\[
 |\ell_w|
 \le\sum_{i<t}w_i\big(|a_i-a_t|+2B\|b_i-b_t\|_2+B^2\|G_i-G_t\|_{\mathrm{op}}\big)
 \le(L_a+2BL_b+B^2L_G)M_w(\psi;t).
\]
By Remark~\ref{rem:smooth-local-sufficient}, $\|(G_i-G_t)\beta_w-(b_i-b_t)\|_2\le(2BL_G+\Gamma L_\beta)\psi_i(t)$ for all $i<t$, hence
\[
 \|G_w^{-1/2}v_w\|_2
 \le\gamma^{-1/2}\|v_w\|_2
 \le\gamma^{-1/2}(2BL_G+\Gamma L_\beta)M_w(\psi;t).
\]
Finally,
\[
 \|G_w^{-1/2}(G_w-G_t)G_w^{-1/2}\|_{\mathrm{op}}
 \le\gamma^{-1}\|G_w-G_t\|_{\mathrm{op}}
 \le\gamma^{-1}L_GM_w(\psi;t).
\]
Combining these bounds with $\varepsilon\le\sqrt\varepsilon$ and Lemma~\ref{lem:linear-excess-decomposition} proves (iii).
\end{proof}

\begin{example}\label{exp:weights}
For prediction at $t$, common weight families \citep{GARDNER2006637,hyndman2008forecasting} use only observations $i<t$:
\begin{enumerate}[label=(\roman*)]
\item Uniform windows: $w_i^{(t)}=s^{-1}\mathbbm1\{t-s\le i<t\}$, $1\le s\le t-1$.
\item Exponential smoothing: $w_i^{(t)}=\exp[-\alpha(t-i)]/\sum_{j=1}^{t-1}\exp[-\alpha(t-j)]$ for $1\le i<t$, and zero otherwise, with $\alpha>0$.
\end{enumerate}
\end{example}

\paragraph{Exponential Weights}
For exponential weights (Example~\ref{exp:weights}(ii)), write $s_\alpha=\min\{(1-e^{-\alpha})^{-1},\,n\}$. Then $\|w\|_2^2\asymp s_\alpha^{-1}$ and $M_w(\psi;t)\lesssim s_\alpha^\nu$, so that the estimation term $p\|w\|_2^2\log^2n\,\log^2(1/\delta)$ is of order $p/s_\alpha$ up to logarithmic factors. Suppose that
\[
 \frac{\sqrt\Gamma}{\gamma}(2BL_G+\Gamma L_\beta)\lesssim\kappa,
 \qquad
 L_a+B^2L_G+2BL_b\lesssim\kappa.
\]
Then Corollary~\ref{cor:excess-smooth}(i) and (ii) give
\[
 R_t(\hat h_{w(\alpha)})-R_t(h_t^*)\lesssim\frac{p}{s_\alpha}+\kappa^2s_\alpha^{2\nu}
 \qquad\text{and}\qquad
 R_t(\hat h_{w(\alpha)})-R_t(h_t^*)\lesssim\frac{p}{s_\alpha}+\kappa s_\alpha^{\nu},
\]
respectively. Their minimizing horizons, truncated at the available history, are
\[
 s_\alpha^*\asymp(p/\kappa^2)^{1/(2\nu+1)}
 \qquad\text{and}\qquad
 \tilde s_\alpha^*\asymp(p/\kappa)^{1/(\nu+1)},
\]
with rates $p^{2\nu/(2\nu+1)}\kappa^{2/(2\nu+1)}$ and $p^{\nu/(\nu+1)}\kappa^{1/(\nu+1)}$. If also $\gamma^{-1/2}(\gamma^{-1/2}L_G+2BL_G+\Gamma L_\beta)\lesssim\kappa$, Corollary~\ref{cor:excess-smooth}(iii) gives the same first-order bound $p/s_\alpha+\kappa s_\alpha^\nu$ for every admissible $\varepsilon\in(0,1]$ satisfying $\bar h_t\in\Hcal_{w,\varepsilon}$, since the term $|\ell_w|$ does not shrink with $\varepsilon$. For interior oracles, $\beta_t-\beta_w=G_t^{-1}v_w$, so the condition $\bar h_t\in\Hcal_{w,\varepsilon}$ holds whenever
\[
 \varepsilon\ge\|G_w\|_{\mathrm{op}}\gamma^{-2}(2BL_G+\Gamma L_\beta)^2M_w(\psi;t)^2.
\]

\section{Nonparametric Regression}\label{app:nonparametric}

We use the notation of Section~\ref{sec:nonparametric-extension} and set
\(a_i=\E_{P_i}[Y^2]\), with \(a_w\), \(G_{w,q}\), and \(b_{w,q}\) denoting
the corresponding weighted moments.
For the generic coefficient class, write
\(B_q=\sup_{\beta\in\Theta_q}\|\beta\|_2\).

\subsection{Drift for Fixed $q$}

\begin{proposition}\label{thm:basis-drift-optimality}
Assume that \(\beta_{w,q}\) and \(\beta_{t,q}\) are interior points of
\(\Theta_q\) and
\[
 \gamma_q I_q\preceq G_{t,q}\preceq\Gamma_q I_q.
\]
Then
\[
 \frac{1}{\Gamma_q}\|v_{w,q}\|_2^2
 \le \|\bar h_{w,q}-\bar h_{t,q}\|_{L_2(P_t)}^2
 \le\frac{1}{\gamma_q}\|v_{w,q}\|_2^2.
\]
\end{proposition}

\begin{proof}
The interior normal equations are
\(G_{w,q}\beta_{w,q}=b_{w,q}\) and
\(G_{t,q}\beta_{t,q}=b_{t,q}\). Therefore,
\[
 v_{w,q}=G_{t,q}(\beta_{t,q}-\beta_{w,q}).
\]
For square loss,
\[
 \|\bar h_{w,q}-\bar h_{t,q}\|_{L_2(P_t)}^2
 =(\beta_{w,q}-\beta_{t,q})^\top
 G_{t,q}(\beta_{w,q}-\beta_{t,q}) = v_{w,q}^\top G_{t,q}^{-1}v_{w,q}.
\]
The eigenvalue bounds complete the proof.
\end{proof}

\begin{proposition}\label{prop:basis-discrepancy-upper}
Write $\operatorname{dis}_{\Hcal_q}(P_w,P_t)$ for the discrepancy restricted to $\Hcal_q$. If \(B_q<\infty\), then
\[
 \operatorname{dis}_{\Hcal_q}(P_w,P_t)
 \le |a_w-a_t|+2B_q\|v_{w,q}\|_2
 +3B_q^2\|G_{w,q}-G_{t,q}\|_{\mathrm{op}}.
\]
\end{proposition}

\begin{proof}
For \(\beta\in\Theta_q\),
\begin{align*}
 R^w(\beta)-R_t(\beta)
 =a_w-a_t-2\beta^\top(b_{w,q}-b_{t,q}) +\beta^\top(G_{w,q}-G_{t,q})\beta.
\end{align*}
Substituting
\(b_{w,q}-b_{t,q}=(G_{w,q}-G_{t,q})\beta_{w,q}-v_{w,q}\), applying
Cauchy--Schwarz, and using
\(\|\beta\|_2,\|\beta_{w,q}\|_2\le B_q\) gives the stated bound uniformly over
\(\beta\in\Theta_q\).
\end{proof}

\begin{proposition}\label{thm:basis-discrepancy}
Suppose \(v_{w,q}\ne0\) and, for some \(\rho_{w,q}>0\),
\[
 \beta_{w,q}\pm\rho_{w,q}
 \frac{v_{w,q}}{\|v_{w,q}\|_2}\in\Theta_q.
\]
If \(0\in\Theta_q\) and, for some \(r_q>0\), \(r_q\mathbb B_2^q\subseteq\Theta_q\), then
\begin{align*}
 \operatorname{dis}_{\Hcal_q}(P_w,P_t)
 \ge\max\Bigg\{&2\rho_{w,q}\|v_{w,q}\|_2,\ |a_w-a_t|, 2r_q\|b_{w,q}-b_{t,q}\|_2,
 \frac{r_q^2}{2}\|G_{w,q}-G_{t,q}\|_{\mathrm{op}}
 \Bigg\}.
\end{align*}
\end{proposition}

\begin{proof}
Let \(f_q(\beta)=R^w(\beta)-R_t(\beta)\). With
\(u=v_{w,q}/\|v_{w,q}\|_2\), direct expansion gives
\[
 f_q(\beta_{w,q}+\rho_{w,q}u)
 -f_q(\beta_{w,q}-\rho_{w,q}u)
 =4\rho_{w,q}\|v_{w,q}\|_2,
\]
so the discrepancy is at least \(2\rho_{w,q}\|v_{w,q}\|_2\). Evaluating at
\(\beta=0\) gives \(|a_w-a_t|\). Next, if \(b_{w,q}\ne b_{t,q}\), for
\(u_b=(b_{w,q}-b_{t,q})/\|b_{w,q}-b_{t,q}\|_2\),
\[
 f_q(r_qu_b)-f_q(-r_qu_b)
 =-4r_q\|b_{w,q}-b_{t,q}\|_2,
\]
which gives the third term; otherwise this term is zero. Finally, let \(u_G\) be a unit eigenvector of
\(G_{w,q}-G_{t,q}\) corresponding to an eigenvalue of maximal absolute value.
Since
\[
 \frac{f_q(r_qu_G)+f_q(-r_qu_G)}{2}
 =a_w-a_t+r_q^2u_G^\top(G_{w,q}-G_{t,q})u_G,
\]
combining this evaluation with \(|f_q(0)|\) gives the final term. Taking the
maximum proves the result.
\end{proof}

\subsection{Weights under Smooth Drift}

\begin{proposition}\label{prop:basis-approximation}
For the equal-width \(q\)-bin histogram class with
\(\Theta_q^{\mathrm{hist}}=[-B,B]^q\), if \(h_t^*\) is \(1\)-Lipschitz and
bounded by \(B\), then
\[
 R_t(\bar h_{t,q})-R_t(h_t^*)\le q^{-2}.
\]
\end{proposition}

\begin{proof}
The conditional mean-zero property of the noise gives
\[
 R_t(h)-R_t(h_t^*)=\|h-h_t^*\|_{L_2(P_t)}^2,
\]
so that
\[
 R_t(\bar h_{t,q})-R_t(h_t^*)
 =\inf_{h\in\Hcal_q}\|h-h_t^*\|_{L_2(P_t)}^2
\]
and, consequently,
\[
 R_t(\bar h_{w,q})-R_t(h_t^*)
 =R_t(\bar h_{w,q})-R_t(\bar h_{t,q}) + \inf_{h\in\Hcal_q}\|h-h_t^*\|_{L_2(P_t)}^2.
\]
For the claim, it therefore suffices to bound the approximation error of a single element of $\Hcal_q^{\mathrm{hist}}$. For the histogram basis, approximate
\(h_t^*\) on each interval by its value at the left endpoint.
The resulting coefficients lie in \([-B,B]\), so this step
function belongs to \(\Hcal_q^{\mathrm{hist}}\). The uniform error
is at most \(q^{-1}\), and hence the squared \(L_2(P_t)\) error is at most
\(q^{-2}\).
\end{proof}

\begin{proposition}\label{prop:basis-estimation}
Suppose the conditions of Theorem~\ref{thm:cr-fast-full} and (S) hold,
\(\sup_z\|\Phi_q(z)\|_2\le1\), and
\(\Theta_q\subseteq B\mathbb B_2^q\). Assume the weight-class and size conditions of
Proposition~\ref{prop:linear-estimation}, with \(p\) replaced by \(q\).
Then, with probability at least \(1-2\delta\), for every \(w\in\Wcal\),
\[
 R^w(\hat h_{w,q})-R^w(\bar h_{w,q})
 \lesssim q\|w\|_2^2\log^2n\,\log^2(1/\delta).
\]
\end{proposition}

\begin{proof}
For \(h_\beta,h_{\beta'}\in\Hcal_q\), \(\|h_\beta-h_{\beta'}\|_\infty\le\|\beta-\beta'\|_2\), so Lemma~\ref{lem:norm-ball-covering} with the Euclidean norm, \(r=B\), and \(d=q\) gives \(\log N_\infty(\epsilon,\Hcal_q)\le q\log(1+2B/\epsilon)\).
Consequently, \eqref{cond:cr-rate-complexities} holds with
\(D\asymp q\) and \(\xi=0\), and Proposition~\ref{prop:cr-rate-simplify} proves the result.
\end{proof}

For the equal-width \(q\)-bin histogram basis, take
\(\Theta_q^{\mathrm{hist}}=[-B,B]^q\) and let \(\Hcal_q^{\mathrm{hist}}\)
be the corresponding class. Since
\(\|h_\beta-h_{\beta'}\|_\infty=\|\beta-\beta'\|_\infty\) and \([-B,B]^q\) is the sup-norm ball of radius \(B\), Lemma~\ref{lem:norm-ball-covering} gives
\[
 \log N_\infty(\epsilon,\Hcal_q^{\mathrm{hist}})
 \le q\log(1+2B/\epsilon).
\]
Thus \eqref{cond:cr-rate-complexities} holds with \(D\asymp q\) and \(\xi=0\), and Proposition~\ref{prop:cr-rate-simplify} gives the bound of Proposition~\ref{prop:basis-estimation} for \(\Hcal_q^{\mathrm{hist}}\). This argument does not use the Euclidean-radius assumption of
Proposition~\ref{prop:basis-estimation}. Define
\[
 L_{G,q}=\max_{i<t}|t-i|^{-\nu}\|G_{i,q}-G_{t,q}\|_{\mathrm{op}},
 \qquad
 L_{b,q}=\max_{i<t}|t-i|^{-\nu}\|b_{i,q}-b_{t,q}\|_2,
\]
and set $D_q=(B_qL_{G,q}+L_{b,q})^2/\gamma_q$.
For \(\Theta_q^{\mathrm{hist}}=[-B,B]^q\),
\(B_q\le B\sqrt q\); this dependence is absorbed into \(D_q\).
Also write $A_t(q)=R_t(\bar h_{t,q})-R_t(h_t^*)$.
We derive the excess-risk bound of Section~\ref{sec:nonparametric-extension} under exponential weights,
\begin{equation}\label{eq:nonparametric-smooth-excess-risk}
  R_t(\hat h_{w(\alpha),q})-R_t(h_t^*)
 \lesssim \frac{q}{s_\alpha}+D_qs_\alpha^{2\nu}+A_t(q).
\end{equation}
\begin{proof}[Derivation of \eqref{eq:nonparametric-smooth-excess-risk}]
Fix $q$ and suppose the class oracles are interior. Since the target oracle is interior, its normal equation implies orthogonality of $h_t^*-\bar h_{t,q}$ to the feature span. Consequently,
\[
 R_t(\hat h_{w,q})-R_t(h_t^*)
 =\|\hat h_{w,q}-\bar h_{t,q}\|_{L_2(P_t)}^2+A_t(q)
 \le 2\|\hat h_{w,q}-\bar h_{w,q}\|_{L_2(P_t)}^2
 +2\|\bar h_{w,q}-\bar h_{t,q}\|_{L_2(P_t)}^2+A_t(q).
\]
By \eqref{cond:cr-loss-upper} and \eqref{cond:cr-bernstein},
\[
 \|\hat h_{w,q}-\bar h_{w,q}\|_{L_2(P_t)}^2
 \le C_\Pcal^2\sum_{i<t}w_i\|\hat h_{w,q}-\bar h_{w,q}\|_{L_2(P_i)}^2
 \le C_\Pcal^2\{R^w(\hat h_{w,q})-R^w(\bar h_{w,q})\}.
\]
By the definitions of $L_{G,q}$ and $L_{b,q}$, $\|\beta_{w,q}\|_2\le B_q$, and the triangle inequality as in the proof of Proposition~\ref{prop:smooth-local},
\[
 \|v_{w,q}\|_2
 \le (B_qL_{G,q}+L_{b,q})M_w(\psi;t),
\]
and Proposition~\ref{thm:basis-drift-optimality} therefore bounds the drift by $(B_qL_{G,q}+L_{b,q})^2M_w(\psi;t)^2/\gamma_q$. Applying Proposition~\ref{prop:basis-estimation}, or the direct histogram bound above for \(\Hcal_q^{\mathrm{hist}}\), to the estimation term yields the full bound
\begin{align*}
 R_t(\hat h_{w,q})-R_t(h_t^*)
 \lesssim q\|w\|_2^2\log^2n\,\log^2(1/\delta) +\frac{(B_qL_{G,q}+L_{b,q})^2}{\gamma_q}M_w(\psi;t)^2+A_t(q).
\end{align*}
Taking $w=w(\alpha)$, using the weight and drift scalings in
Section~\ref{sec:excess-risk-linear}, gives \eqref{eq:nonparametric-smooth-excess-risk}.
\end{proof}

For the one-dimensional Lipschitz histogram class,
Proposition~\ref{prop:basis-approximation} gives $A_t(q)\le q^{-2}$.
In the illustrative regime $D_q\asymp \bar D$, minimizing the three leading terms yields
\[
 s_\alpha^*\asymp \bar D^{-3/(6\nu+2)},\qquad
 q^*\asymp \bar D^{-1/(6\nu+2)},
 \qquad
 R_t(\hat h_{w(\alpha^*),q^*})-R_t(h_t^*)\lesssim \bar D^{1/(3\nu+1)}.
\]

\section{Simulation}
\label{app:simulation-setup}

The Gaussian examples below illustrate the deterministic drift analysis and exact risk trade-offs.

\subsection{Scalar Examples}

The illustration in Figure~\ref{fig:drift-error-illustrations} uses square loss
and the scalar hypothesis class
\[
    \Hcal=\{h_\beta:\ |\beta|\le B\}, \qquad B=4,
\]
where \(h_\beta\) is a constant predictor, and we write \(t=0,\ldots,T\) with \(T=200\).

\paragraph{Noise Drift}
The first setting is a mean-stationary process with $\beta_i^*=0$ for all $i=0,\ldots,T$ and time-varying noise
variance,
\[
    Y_t=\eta_t, \qquad
    \eta_t\sim N(0,\sigma_t^2), \qquad
    \sigma_t^2=0.12+2.6\left(\frac{t}{T}\right)^2 .
\]
Thus the target oracle is \(\beta_t^*=0\) at every time point.
We use exponential smoothing weights
\[
    w_i^{(t)}
    =
    \frac{\exp\{-\alpha(t-i)\}}
    {\sum_{j=0}^{t-1}\exp\{-\alpha(t-j)\}},
    \qquad i=0,\ldots,t-1,
\]
with \(\alpha=0.045\).

\paragraph{Smooth Drift}
The second setting is a scalar smoothly drifting regression function,
\[
    Y_t=\beta_t^*+\eta_t, \qquad
    \eta_t\sim N(0,\sigma^2), \qquad
    \sigma^2=0.4,
\]
where
\[
    \beta_t^*
    =
    -0.85+1.65\frac{t}{T}
    +0.20\sin\left(4\pi\frac{t}{T}\right).
\]
For each \(t=1,\ldots,T\), the plotted smooth-drift terms use
exponential smoothing weights on the observations before \(t\),
\[
    w_i^{(t)}
    =
    \frac{\exp\{-\alpha(t-i)\}}
    {\sum_{j=0}^{t-1}\exp\{-\alpha(t-j)\}},
    \qquad i=0,\ldots,t-1,
\]
with \(\alpha=0.012\). The weighted oracle is
\(\beta_w^{(t)}=\sum_{i=0}^{t-1}w_i^{(t)}\beta_i^*\), and the
oracle drift at time \(t\) is
\[
    R_t(\bar h_w)-R_t(\bar h_t)
    =
    (\beta_w^{(t)}-\beta_t^*)^2 .
\]
Writing \(a_t=\E[Y_t^2]=(\beta_t^*)^2+\sigma^2\) and
\(b_t=\E[Y_t]=\beta_t^*\), the discrepancy is
\[
    \operatorname{dis}(P_w^{(t)},P_t)
    =
    \sup_{|\beta|\le B}
    \left|
        (a_w^{(t)}-a_t)-2(b_w^{(t)}-b_t)\beta
    \right|,
\]
where \(a_w^{(t)}=\sum_{i=0}^{t-1} w_i^{(t)}a_i\) and
\(b_w^{(t)}=\sum_{i=0}^{t-1} w_i^{(t)}b_i\). The relevant temporal bias is
\[
    M_w(\psi;t)=\sum_{i=0}^{t-1}w_i^{(t)}(t-i),
\]
corresponding to \(\psi_i(t)=t-i\).

The last panel of Figure~\ref{fig:drift-error-illustrations} plots the exact
discrepancy and oracle drift as functions of
\(t\).

For Figure~\ref{fig:smooth-bias-variance-tradeoff}, we fix the final target
time \(t=T\) and vary the exponential-weight parameter \(\alpha\).
For the weighted estimator
\(\hat\beta_w^{(T)}=\sum_{i=0}^{T-1}w_i^{(T)}Y_i\), the oracle drift error
and estimation error are
\[
    D_{\mathrm{od}}(w)
    =
    \left(\beta_w^{(T)}-\beta_T^*\right)^2,
    \qquad
    E_{\mathrm{est}}(w)
    =
    \E\left[(\hat\beta_w^{(T)}-\beta_w^{(T)})^2\right]
    =
    \sigma^2\sum_{i=0}^{T-1}\left(w_i^{(T)}\right)^2,
\]
and
\[
    \E\left[(\hat\beta_w^{(T)}-\beta_T^*)^2\right]
    =
    D_{\mathrm{od}}(w)+E_{\mathrm{est}}(w).
\]

\subsection{Basis Expansions}
Let $Z_t\sim\mathrm{Unif}[0,1)$ be stationary and
\[
 Y_t=h_t^*(Z_t)+\eta_t,\qquad \eta_t\sim N(0,0.4),\qquad h_t^*(z)=0.8z+\sin\bigl(1.5\pi(z+0.8\,t/T)\bigr),
\]
with $T=200$. We predict at $t=T$ from the observations $0,\ldots,T-1$ with exponential weights, $\alpha=0.03$, so that $n_{\mathrm{eff}}\approx66$. For each $q$, $\Hcal_q$ is the equal-width $q$-bin histogram class with a coefficient bound $B\ge2$, so all class oracles are binwise averages. Because the feature distribution is stationary, $\bar h_{w,q}$ and $\bar h_{t,q}$ are the binwise averages of $\sum_iw_ih_i^*$ and $h_T^*$, and the oracle drift and approximation error are computed exactly by quadrature. The estimation error is the first-order variance of the binwise weighted means,
\[
 \sum_{j=1}^q\frac1q\,
 \frac{\sum_{i<T}w_i^2\operatorname{Var}\bigl(\mathbbm1\{Z_i\in I_j\}(Y_i-\beta_{w,q,j})\bigr)}{q^{-2}},
\]
and the plotted excess risk is the sum of the three terms.

\end{document}